\documentclass[a4paper,12pt,reqno]{amsart}
\usepackage[T1]{fontenc}
\usepackage[utf8]{inputenc}
\usepackage{amssymb,bbm,mathrsfs}
\usepackage[margin=1in]{geometry}
\usepackage{enumitem}
\usepackage[bookmarksdepth=2]{hyperref}

\numberwithin{equation}{section}

\newtheorem{Theorem}{Theorem}
\newtheorem{theorem}{Theorem}[section]
\newtheorem{lemma}[theorem]{Lemma}
\newtheorem{corollary}[theorem]{Corollary}

\theoremstyle{definition}
\newtheorem{definition}[theorem]{Definition}

\DeclareMathOperator{\re}{Re}
\DeclareMathOperator{\im}{Im}

\newcommand{\leb}{\mathscr{L}}
\newcommand{\cont}{\mathscr{C}}
\newcommand{\dom}{\mathscr{D}}
\newcommand{\subdom}{\mathscr{D}}
\newcommand{\form}{\mathscr{E}}
\newcommand{\fourier}{\mathscr{F}}
\newcommand{\op}{L}
\newcommand{\harm}{{\mathrm{harm}}}
\newcommand{\rad}{{\mathrm{rad}}}

\newcommand{\R}{\mathbf{R}}
\newcommand{\N}{\mathbf{N}}
\newcommand{\hh}{\mathbf{H}}
\newcommand{\ind}{\mathbf{1}}
\newcommand{\ph}{\varphi}
\newcommand{\eps}{\varepsilon}
\newcommand{\thet}{\vartheta}

\newcommand{\tnorm}[1]{\lVert #1 \rVert}

\newcommand{\tscalar}[1]{\langle #1 \rangle}

\newcommand{\expr}[1]{\left( #1 \right)}

\DeclareMathOperator{\ext}{ext}
\DeclareMathOperator{\extt}{e\widetilde{x}t}

\renewcommand{\le}{\leqslant}
\renewcommand{\ge}{\geqslant}

\usepackage{ifthen}
\newcommand{\formula}[2][nolabel]%
{%
 \ifthenelse{\equal{#1}{nolabel}}%
 {\begin{align*} #2 \end{align*}}%
 {%
  \ifthenelse{\equal{#1}{}}%
  {\begin{align} #2 \end{align}}%
  {\begin{align} \label{#1} #2 \end{align}}%
 }%
}

\theoremstyle{definition}

\usepackage{ifthen}

\newcommand{\C}{\mathbb{C}}

\newcommand{\sub}{\subseteq}

\theoremstyle{definition}

\begin{document}

\title[Extension technique]{Extension technique for complete Bernstein functions of the Laplace operator}
\author{Mateusz Kwaśnicki, Jacek Mucha}
\thanks{Work supported by the Polish National Science Centre (NCN) grant no.\@ 2015/19/B/ST1/01457}
\address{Mateusz Kwaśnicki, Jacek Mucha \\ Faculty of Pure and Applied Mathematics \\ Wrocław University of Science and Technology \\ ul. Wybrzeże Wyspiańskiego 27 \\ 50-370 Wrocław, Poland}
\email{mateusz.kwasnicki@pwr.edu.pl}
\email{jacek.mucha@pwr.edu.pl}
\date{\today}

\begin{abstract}
We discuss representation of certain functions of the Laplace operator~$\Delta$ as Dirichlet-to-Neumann maps for appropriate elliptic operators in half-space. A~classical result identifies $(-\Delta)^{1/2}$, the square root of the $d$-dimensional Laplace operator, with the Dirichlet-to-Neumann map for the $(d + 1)$-dimensional Laplace operator $\Delta_{t,x}$ in $(0, \infty) \times \R^d$. Caffarelli and Silvestre extended this to fractional powers $(-\Delta)^{\alpha/2}$, which correspond to operators $\nabla_{t,x} (t^{1 - \alpha} \nabla_{t,x})$. We provide an analogous result for all complete Bernstein functions of $-\Delta$ using Krein's spectral theory of strings.

Two sample applications are provided: a Courant--Hilbert nodal line theorem for harmonic extensions of the eigenfunctions of non-local Schrödinger operators $\psi(-\Delta) + V(x)$, as well as an upper bound for the eigenvalues of these operators. Here $\psi$ is a complete Bernstein function and $V$ is a confining potential.
\end{abstract}

\maketitle

%
%

\section{Introduction}
\label{sec:intro}

A classical result states that the Dirichlet-to-Neumann operator in half-space is the square root of the Laplace operator; namely, if $u(t, x)$ is harmonic in $\hh = (0, \infty) \times \R^d$ with boundary value $f(x) = u(0, x)$, then, given some boundedness condition on $u$, we~have
\formula{
 -(-\Delta)^{1/2} f(x) & = \partial_t u(0, x) .
}
The above observation was extended to general fractional powers of the Laplace operator in the context of non-local partial differential equations by Caffarelli and Silvestre in~\cite{bib:cs07}: for~$\alpha \in (0, 2)$, if $u$ satisfies the elliptic equation
\formula[eq:oph:cs]{
 \nabla_{t,x} (t^{1 - \alpha} \nabla_{t,x} u(t, x)) & = 0
}
in $\hh$ with boundary value $f(x) = u(0, x)$, then, under appropriate boundedness assumption on $u$, we have
\formula[eq:op:cs]{
\begin{aligned}
 -(-\Delta)^{\alpha/2} f(x) & = \frac{|\Gamma(-\tfrac{\alpha}{2})|}{\alpha 2^\alpha \Gamma(\tfrac{\alpha}{2})} \lim_{t \to 0^+} t^{1 - \alpha} \partial_t u(t, x) \\
 & = \frac{|\Gamma(-\tfrac{\alpha}{2})|}{2^\alpha \Gamma(\tfrac{\alpha}{2})} \lim_{t \to 0^+} \frac{u(t, x) - u(0, x)}{t^\alpha} \, .
\end{aligned}
}
Noteworthy, the above extension dates back to the paper of Molchanov and Ostrovski~\cite{bib:mo69} within the probabilistic context, and it has been used by other authors before the work of Caffarelli and Silvestre; see, for example~\cite{bib:d90,bib:m87,bib:gz03}. Nevertheless, the representation of the fractional Laplace operator given in~\eqref{eq:op:cs} is now known as the Caffarelli--Silvestre extension technique. It is most commonly applied in the setting of $\leb^2$ spaces, but versions for other Banach spaces or operators are also available, see~\cite{bib:atw18,bib:gms13,bib:st10}. In fact, the above technique works for fractional powers of essentially arbitrary nonnegative self-adjoint operators. For a brief discussion and further properties, we refer to Section~2.8 in~\cite{bib:k17}.

It is a relatively simple consequence of \emph{Krein's spectral theory of strings} that if the weight $t^{1 - \alpha}$ in~\eqref{eq:oph:cs} is replaced by a more general function $a(t)$, then the corresponding Dirichlet-to-Neumann operator is again a function of $-\Delta$, say $\psi(-\Delta)$, and furthermore $\psi$ is a \emph{complete Bernstein function}. Conversely, any complete Bernstein function of $-\Delta$ can be represented in this way, if one allows for certain singularities of $a(t)$. Even though this method has already been discussed in literature (see, for example,~\cite{bib:fls15,bib:k11,bib:wiki}), finding a good reference is problematic.

The main purpose of this article is to fill in this gap and discuss rigorously the above-mentioned general extension technique. This is complemented by two applications for non-local Schrödinger operators $\psi(-\Delta) + V(x)$: a Courant--Hilbert theorem on nodal domains (for the extension problem), and an upper bound for the eigenvalues. For simplicity, we focus on $L^2(\R^d)$ results, although a more general approach seems to be possible.

We remark that the extension technique can also be generalised in different directions. For example, higher-order powers of $-\Delta$ can be studied in a somewhat similar way, see~\cite{bib:cg11,bib:gs18,bib:rs16,bib:y18}. Furthermore, in~\cite{bib:fgmt15} a closely related, but essentially different extension technique is developed in a non-commutative setting, for the sub-Laplacian on the Heisenberg group.

\subsection{Extension technique}

If $\psi$ is a function on $[0, \infty)$, by $\psi(-\Delta)$ we understand the Fourier multiplier on $\leb^2(\R^d)$ with symbol $\psi(|\xi|^2)$, that is, the operator
\formula[eq:op]{
 \psi(-\Delta) f & = \fourier^{-1} (\psi(|\cdot|^2) \fourier f(|\cdot|)) .
}
Here $\fourier$ denotes the Fourier transform, and $f$ is in the domain $\dom(\psi(-\Delta))$ of $\psi(-\Delta)$ if both $f(\xi)$ and $\psi(|\xi|^2) \fourier f(\xi)$ are in $\leb^2(\R^d)$. Our goal is to identify the operator $\psi(-\Delta)$, for appropriate functions $\psi$, with the Dirichlet-to-Neumann operator for an appropriate elliptic differential operator in $\hh$.

The differential operators in $\hh$ that we consider are of the form
\formula[eq:oph]{
 \nabla_{t,x} (a(t) \nabla_{t,x} u(t, x)) ,
}
sometimes expressed in an equivalent way as
\formula[eq:oph:t]{
 \Delta_{t,x} u(t, x) + \frac{a'(t)}{a(t)} \, \partial_t u(t, x) & = \frac{1}{a(t)} \, \nabla_{t,x} (a(t) \nabla_{t,x} u(t, x)) ,
}
or, in variable $s$ such that $ds = (a(t))^{-1} dt$,
\formula[eq:oph:s]{
 \partial_s^2 u(s, x) + A(s) \Delta_x u(s, x) .
}
Here $a(t)$ and $A(s)$ are nonnegative coefficients, related one to the other by the condition $A(s) ds = a(t) dt$. For further details about this change of variable, see Section~\ref{sec:krein}. In fact we allow the coefficient $A(s)$ to be a locally finite measure $A(ds)$ on some interval $[0, R)$, with $R$ is possibly infinite; in this case the corresponding differential operator is defined on $(0, R) \times \R^d$. To keep the presentation simple, however, we discuss this general case only in Appendix~\ref{app:krein}, and (with minor exceptions) in the remaining part of the article we only consider the \emph{regular} case, when $R = \infty$ and $A(s)$ is a nonnegative, locally integrable function.

We consider functions $u(t, x)$ (or $u(s, x)$) harmonic with respect to any of the above differential operators. The Dirichlet-to-Neumann operator is defined as
\formula[eq:op:t]{
 \op f(x) & = \lim_{t \to 0^+} \frac{u(t, x) - u(0, x)}{s(t)} = \lim_{t \to 0^+} a(t) \partial_t u(t, x)
}
for the operator~\eqref{eq:oph:t} (where $s$ is again given by $ds = (a(t))^{-1} dt$), and
\formula[eq:op:s]{
 \op f(x) & = \partial_s u(0, x) = \lim_{s \to 0^+} \frac{u(s, x) - u(0, x)}{s}
}
when the form given in~\eqref{eq:oph:s} is considered.

The main result in this part is contained in the following theorem, which is a relatively simple corollary of Krein's spectral theory of strings.

\begin{Theorem}\label{Thm:op}
Given the coefficient $a(t)$ or $A(s)$ \textnormal{(}or, more generally, $A(ds)$\textnormal{)}, there is a complete Bernstein function $\psi$ such that $\op = \psi(-\Delta)$. Conversely, for any complete Bernstein function $\psi$ there is a unique corresponding coefficient $A(ds)$.
\end{Theorem}

The above result, as well as many other results in this section, is stated in a somewhat informal way: we do not specify the conditions on $A(s)$ or $a(t)$, nor we give a rigorous definition of $L$. In fact, Theorem~\ref{Thm:op} is a combination of several results: bijective correspondence between $A(s)$ and $\psi$ is a part of Theorem~\ref{thm:krein}; the identification of $L$ and $\psi(-\Delta)$ in the regular case is given in Theorems~\ref{thm:oph:s} and~\ref{thm:harm} for the operator~\eqref{eq:oph:s}, and in Theorem~\ref{thm:oph:t} for the operator~\eqref{eq:oph:t}; the general case is discussed in Theorem~\ref{thm:harm:G}. These results are carefully stated, and include all necessary definitions and assumptions.

Noteworthy, the \emph{Krein correspondence} between $\psi$ and $a(t)$ or $A(s)$ described in Theorem~\ref{Thm:op} is not explicit, and there is no easy way to find the coefficients $a(t)$ or $A(s)$ corresponding to a given complete Bernstein function $\psi$. Furthermore, only a handful of explicit pairs of corresponding $\psi$ and $a(t)$ or $A(s)$ are known.

It is often more convenient to have the identification described in Theorem~\ref{Thm:op} at the level of quadratic forms. This is discussed in our next result, which involves the quadratic form $\form_H(u, u)$, defined by
\formula[eq:form:t]{
 \form_H(u, u) & = \int_0^\infty \int_{\R^d} a(t) |\nabla_{t,x} u(t, x)|^2 dx dt
}
for the operator~\eqref{eq:oph:t}, and by
\formula[eq:form:s]{
 \form_H(u, u) & = \int_0^\infty \int_{\R^d} \bigl((\partial_s u(s, x))^2 + A(s) |\nabla_x u(s, x)|^2\bigr) dx ds
}
for the operator~\eqref{eq:oph:s}.

\begin{Theorem}\label{Thm:form}
If $u(0, x) = f(x)$, then
\formula{
 \form_H(u, u) & \ge \int_{\R^d} \psi(|\xi|^2) |\fourier f(\xi)|^2 d\xi .
}
Furthermore, equality holds if and only if either side is finite and $u$ is harmonic with respect to the corresponding operator~\eqref{eq:oph:t} or~\eqref{eq:oph:s}.
\end{Theorem}

In the regular case, a detailed statement of Theorem~\ref{Thm:form}, which includes all necessary assumptions on $u$ and $f$, is given in Theorems~\ref{thm:forms:s} and~\ref{thm:min} for the operator~\eqref{eq:oph:s}, and in Theorem~\ref{thm:forms:t} for the operator~\eqref{eq:oph:t}. The general case is studied in Theorem~\ref{thm:forms:G}.

\subsection{Non-local Schrödinger operators}

The eigenvalues and eigenfunctions of the non-local Schrödinger operator $L + V(x) = \psi(-\Delta) + V(x)$ admit a standard variational description in terms of the quadratic form
\formula{
 \int_{\R^d} \psi(|\xi|^2) |\fourier f(\xi)|^2 d\xi + \int_{\R^d} V(x) (f(x))^2 dx .
}
More precisely, the $n$-th eigenfunction is the minimiser of the above expression among all functions $f$ with $\leb^2(\R^d)$ norm $1$ that are orthogonal to the preceding $n - 1$ eigenfunctions. (Here and below eigenfunctions are always arranged so that the corresponding eigenvalues for a non-decreasing sequence. We assume that $V$ is a confining potential, so that $L + V(x)$ has purely discrete spectrum).

Theorem~\ref{Thm:form} implies that if $\psi$ is a complete Bernstein function, then the above quadratic form can be replaced by the local expression
\formula{
 \form_H(u, u) + \int_{\R^d} V(x) (u(0, x))^2 dx .
}
In this case the $n$-th eigenfunction is the boundary value of the minimiser of the above quadratic form among all functions $u$ such that the boundary value $u(0, \cdot)$ is orthogonal to the preceding $n - 1$ eigenfunctions and has $\leb^2(\R^n)$ norm $1$. For a formal statement, we refer to Theorem~\ref{thm:var:h}.

Standard arguments show that, given minimal regularity of $A(s)$, the harmonic extension of the $n$-th eigenfunction can have no more than $n$ nodal parts.

\begin{Theorem}\label{Thm:ch}
Suppose that the coefficient $a(t)$ or $A(s)$ is positive and locally Lipschitz continuous, and that $V$ is a locally bounded confining potential. Let $f_n$ be the $n$-th eigenfunction of the operator $\op + V(x)$, and let $u_n$ be the extension of $f_n$ to $\hh$ which is harmonic with respect to the operator $\op_H$. Then $u_n$ has no more than $n$ nodal parts.
\end{Theorem}

A formal statement is given in Theorem~\ref{thm:ch}. By a simple geometric argument, Theorem~\ref{Thm:ch} implies that if $d = 1$, then $f_n$ has no more than $2 n - 1$ nodal parts. A natural conjecture states that in fact $f_n$ has no more than $n$ nodal parts, in fact --- also in higher dimensions; this is however an open problem even for the fractional Laplace operator with very simple potentials, for example, an infinite potential well. For further discussion, see~\cite{bib:bk04}.

If $d \ge 2$, Theorem~\ref{Thm:ch} does not provide any bound on the number of nodal parts of $f_n$. However, if $V(x)$ is a radial function, then the operator $\op + V(x)$ preserves the class of radial functions, and Theorem~\ref{thm:ch} can be applied to this restriction. This leads to the following result, which gives a bound on the number of nodal parts of radial eigenfunctions. This bound is still rather unsatisfactory, but it is applicable; see, for example,~\cite{bib:fl13,bib:fls15}.

\begin{Theorem}\label{Thm:ch:rad}
Suppose that $d \ge 2$, that the coefficient $a(t)$ or $A(s)$ is positive and locally Lipschitz continuous, and that $V$ is a radial, locally bounded confining potential. Let $f_{\rad,n}$ be the $n$-th radial eigenfunction of the operator $\op + V(x)$, and let $u_{\rad,n}$ be the extension of $f_{\rad,n}$ to $\hh$ which is harmonic with respect to the operator $\op_H$. Then $u_{\rad,n}$ has no more than $n$ nodal parts, and $f_{\rad,n}$ has no more than $2 n - 1$ nodal parts.
\end{Theorem}

For a rigorous statement we refer to Theorem~\ref{thm:ch:rad}. Again, it is conjectured that the number of nodal parts of $f_{\rad,n}$ in fact does not exceed $n$.

Curiously, to our best knowledge, given a complete Bernstein function $\psi$, it is not known whether there is a simpler way to verify that the coefficients $a(t)$ or $A(s)$ are locally Lipschitz continuous other than finding these coefficients explicitly.

Theorems~\ref{Thm:ch} and~\ref{Thm:ch:rad} are extensions of the results of Section~3.2 in~\cite{bib:fl13} or Section~5 in~\cite{bib:fls15}, where $\psi(\lambda) = \lambda^{\alpha/2}$ is studied. A rather simple modification shows that the assertion of Theorem~\ref{Thm:ch} holds true also when $V$ is a potential well, that is, for the operator $\psi(-\Delta)$ in a domain, with zero exterior condition. This modifications was studied, for $\psi(\lambda) = \lambda^{1/2}$, in~\cite{bib:bk04}.

Our last main result provides an upper bound for the eigenvalues $\mu_n$ of the non-local Schrödinger operator $\op + V(x)$ in terms of the eigenvalues $\lambda_n$ of a standard Schrödinger operator $-\Delta + \gamma V(x)$ for an appropriate constant $\gamma$ (depending on $n$). For simplicity, below we state the result for the fractional Laplace operator and homogeneous potentials, which is identical to Corollary~\ref{cor:est:fr:hom}. For a more general version, we refer to Theorem~\ref{thm:est}.

\begin{Theorem}\label{Thm:est}
Suppose that $V$ is a locally bounded, positive (except at zero) potential which is homogeneous with degree $p > 0$. Let $\lambda_n$ be the eigenvalues of $-\Delta + V(x)$, and let $\mu_n$ be the eigenvalues of $(-\Delta)^{\alpha/2} + V(x)$. Then
\formula{
 \mu_n & \le \lambda_n^{(2 + p) \alpha / (2 \alpha + 2 p)} .
}
\end{Theorem}

As $p \to \infty$, the above bound approximates the well-known bound $\mu_n \le \lambda_n^{\alpha/2}$ for the eigenvalues of $(-\Delta)^{\alpha/2}$ and $-\Delta$ in a domain, with zero boundary/exterior condition. This was proved (for general complete Bernstein functions $\psi$) by DeBlassie~\cite{bib:d04} and Chen and Song~\cite{bib:cs05}.

An estimate related to that of Theorem~\ref{Thm:est} is proved recently in~\cite{bib:jw18}. Further results on spectral theory of non-local Schrödinger operators can be found, among others, in~\cite{bib:cms90,bib:f18,bib:fls08,bib:kk10,bib:kl15,bib:kl17,bib:kkms10,bib:k09,bib:k12,bib:ly88,bib:lm12}.

\bigskip

The structure of the article corresponds to the above description of main results. We begin with a short section on preliminary results. Then we outline Krein's spectral theory of strings in Section~\ref{sec:krein}; a more in-depth discussion is deferred to Appendix~\ref{app:krein}. In Section~\ref{sec:h} we introduce the harmonic extension techinque and prove Theorems~\ref{Thm:op} and Theorem~\ref{Thm:form}. After that we discuss a number of examples in Section~\ref{sec:ex}. Next three sections discuss applications to non-local Schrödinger operators: the variational principle (Section~\ref{sec:var}), the Courant--Hilbert theorem (Section~\ref{sec:ch}) and estimates of eigenvalues (Section~\ref{sec:eig}). We conclude the paper with a brief discussion of probabilistic aspects of the extension method.

%
%

\section{Preliminaries}

In this section we collect definitions and known results for later use. Throughout the text $d = 1, 2, \ldots$ denotes the dimension. We typically use letters $f, g, h$ for functions defined on $\R^d$, and letters $u, v, w$ for functions on the half-space $\hh = (0, \infty) \times \R^d$. We denote the Lebesgue space of $p$-integrable functions on a domain $D$ by $\leb^p(D)$. By $\fourier f$ we denote the isometric Fourier transform of a function $f$: if $f \in \leb^1(\R^d)$, then
\formula{
 \fourier f(\xi) & = (2 \pi)^{-d/2} \int_{\R^d} e^{-i \xi \cdot x} f(x) dx ,
}
and $\fourier$ extends to a unitary operator on $\leb^2(\R^d)$.

\subsection{Complete Bernstein functions}

Functions $\psi : [0, \infty) \to [0, \infty)$ of the form
\formula{
 \psi(\lambda) & = c_1 + c_2 \lambda + \frac{1}{\pi} \int_{(0, \infty)} \frac{\lambda}{\lambda + s} \, \frac{m(ds)}{s}
}
for some constants $c_1, c_2 \ge 0$ and some nonnegative measure $m$ that satisfies the integrability condition $\int_{(0, \infty)} s^{-1} (1 + s)^{-1} m(ds) < \infty$, are said to be \emph{complete Bernstein functions}. This class has found numerous applications in various areas of mathematics, and it admits several characterisations, of which we mention two:
\begin{enumerate}
\item A function $\psi : [0, \infty) \to [0, \infty)$ is a complete Bernstein function if and only if it extends to a holomorphic function of $\lambda \in \C \setminus (-\infty, 0]$, with $\im \psi(\lambda) \ge 0$ whenever $\im \lambda \ge 0$ (that is, $\psi$ is a Pick function which is nonnegative on $(0, \infty)$).
\item The class of complete Bernstein functions coincides with the class of nonnegative \emph{operator monotone functions}, that is, functions $\psi$ such that $\psi(\op_2) - \psi(\op_1)$ is nonnegative definite whenever $\op_1, \op_2$ are self-adjoint operators such that $\op_2 - \op_1$ is nonnegative definite.
\end{enumerate}
Yet another characterisation, which is crucial for our needs, is given by Krein's correspondence, as stated in Theorem~\ref{thm:krein}. For a detailed discussion of complete Bernstein functions we refer to~\cite{bib:ssv12}.

\subsection{Weak differentiability and ACL property}

As usual, by $\partial_j$ we denote the derivative along $j$-th coordinate, by $\nabla$ the gradient, and by $\Delta$ the Laplace operator. We use the same symbol $\partial_j$ for pointwise and weak derivatives. Recall that a~locally integrable function $f$ defined on a domain $D \sub \R^d$ is weakly differentiable if and only if there exist functions $\partial_j f$, $j = 1, 2, \ldots, d$, such that
\formula{
 \int_D \partial_j f(x) g(x) dx & = -\int_D f(x) \partial_j g(x) dx ,
}
for all $g \in C_c^\infty(D)$ (infinitely smooth, compactly supported functions), where $\partial_j g$ is the usual derivative of $g$.

For brevity, by an absolutely continuous function we will always mean a locally absolutely continuous function. In dimension one, $f$ is weakly differentiable if and only if it is equal almost everywhere to an absolutely continuous function $\tilde{f}$. In this case $\tilde{f}$ is differentiable almost everywhere, $\tilde{f}(y) - \tilde{f}(x)$ it is equal to the integral of $\tilde{f}'$ over $[x, y]$, and the weak derivative of $f$ is equal (almost everywhere) to $\tilde{f}'$.

A similar description is available in higher dimensions, using absolute continuity on lines, abbreviated as~ACL. Namely, a function $f$ has the~ACL property if, for every cardinal direction, $f$ is absolutely continuous on almost every line in that direction. A~well-known theorem asserts that $f$ is weakly differentiable if and only if it is equal almost everywhere to a function $\tilde{f}$ with the~ACL property such that $\nabla \tilde{f}(x)$ is locally integrable in $D$.

%
%

\section{Krein's spectral theory of strings}
\label{sec:krein}

\subsection{Fundamental result}
\label{sec:krein:thm}

The harmonic extension technique is based on Krein's spectral theory of strings. Our starting point is the following theorem, which is essentially due to Krein. References and further discussion can be found in Appendix~\ref{app:krein}.

\begin{theorem}
\label{thm:krein}
Suppose that $A(s)$ is a locally integrable function on $[0, \infty)$ with values in $[0, \infty]$. Then for every $\lambda \ge 0$ there exists a unique non-increasing function $\ph_\lambda(s)$ on $[0, \infty)$ which solves
\formula[eq:ph]{
 \begin{cases}
 \ph_\lambda''(s) = \lambda A(s) \ph_\lambda(s) & \text{for $s > 0$,} \\
 \ph_\lambda(0) = 1 , \\
 \lim_{s \to \infty} \ph_\lambda(s) \ge 0
 \end{cases}
}
(with the second derivative understood in the weak sense). Furthermore, the expression
\formula{
 \psi(\lambda) & = -\ph_\lambda'(0)
}
defines a complete Bernstein function $\psi$, and the correspondence between $A(s)$ and $\psi(\lambda)$ is one-to-one.

More generally, let $A$ be a \emph{Krein's string}: a locally finite nonnegative Borel measure $A(ds)$ on $[0, R)$, where $R \in (0, \infty]$. Then for every $\lambda \ge 0$ there is a unique non-increasing function $\ph_\lambda(s)$ on $[0, R)$ which solves problem~\eqref{eq:ph} (with the second derivative understood in the sense of distributions), with an additional Dirichlet boundary condition $\ph_\lambda(R) = 0$ at $s = R$ imposed in the case when $R$ is finite and $(R - s) A(ds)$ is a finite measure.

Furthermore, any non-zero complete Bernstein function $\psi$ can be obtained in the above manner in a unique way.
\end{theorem}

We remark that the Neumann boundary condition $\ph_\lambda'(R) = 0$ would correspond to extending the string to $[0, \infty)$ in such a way that $A = 0$ on $[R, \infty)$. If $(R - s) A(ds)$ is an infinite measure on a finite interval $[0, R)$, then the Dirichlet condition at $s = R$ is automatically satisfied.

We also note that when $R = \infty$ or when $(R - s) A(ds)$ is an infinite measure on a finite interval $[0, R)$, then any solution of the system~\eqref{eq:ph} which is not a multiple of $\ph_\lambda$ necessarily diverges to $\infty$ or $-\infty$ as $s \to \infty$.

In order not to overwhelm the reader with technical details, throughout the article we restrict our attention to the \emph{regular} case, that is, we will assume that $A$ is a nonnegative locally integrable function on $[0, \infty)$. The extension to the general case is discussed in Appendix~\ref{app:krein}, where also the existence and properties of $\ph_\lambda$ are further discussed. More information about Krein's spectral theory of strings can be found, for example, in Chapter~5 of~\cite{bib:dm76}, in~\cite{bib:kw82} and in Chapter~15 of~\cite{bib:ssv12}.

Note that $\ph_\lambda$ is non-increasing and convex, so that in particular
\formula[eq:ph:est]{
 0 \le -\ph_\lambda'(s) & \le \frac{\ph_\lambda(0) - \ph_\lambda(s)}{s} \le \frac{1}{s}
}
for all $s > 0$ and $\lambda \ge 0$. Although this estimate is very rough, it is completely sufficient for our needs.

We will need the following rather standard property: if $f$ is an absolutely continuous function satisfying $f(0) = 1$, then
\formula[eq:ph:min]{
 \int_0^\infty \bigl((f'(s))^2 + \lambda A(s) (f(s))^2\bigl) ds & \ge \psi(\lambda) ,
}
and equality holds if and only if $f = \ph_\lambda$. For completeness, we provide a short proof.

Equality for $f = \ph_\lambda$ follows by integration by parts: we have $\ph_\lambda(0) \ph_\lambda'(0) = -\psi(\lambda)$, and $\ph_\lambda(s) \ph_\lambda'(s)$ converges to zero as $s \to \infty$, so that
\formula{
 \int_0^\infty (\ph_\lambda'(s))^2 ds & = \psi(\lambda) - \int_0^\infty \ph_\lambda(s) \ph_\lambda''(s) ds .
}
Since $\ph_\lambda''(s) = \lambda A(s) \ph_\lambda(s)$, equality in~\eqref{eq:ph:min} for $f = \ph_\lambda$ follows.

Suppose now that $f$ is absolutely continuous and $f(0) = 1$. Estimate~\eqref{eq:ph:min} holds trivially if $f'$ is not square integrable. Otherwise, by Schwartz inequality, for $S > 0$ we have
\formula[eq:ph:parts]{
 |f(S) - f(0)|^2 & \le \expr{\int_0^S |f'(s)| ds}^2 \le S \int_0^S |f'(s)|^2 ds \le S \int_0^\infty |f'(s)|^2 ds .
}
Therefore, $|f(S)| \le 1 + \sqrt{S} \int_0^\infty |f'(s)|^2 ds$ for all $S > 0$. In particular, by~\eqref{eq:ph:est}, $f(s) \ph_\lambda'(s)$ converges to zero as $s \to \infty$. Denote $g(s) = f(s) - \ph_\lambda(s)$. Then $g$ is an absolutely continuous function, $g(0) = 0$ and $g(s) \ph_\lambda'(s)$ converges to zero as $s \to \infty$. Therefore, integration by parts gives
\formula{
 \int_0^\infty g'(s) \ph_\lambda'(s) ds & = -\int_0^\infty g(s) \ph_\lambda''(s) ds = -\int_0^\infty \lambda A(s) g(s) \ph_\lambda(s) ds .
}
Clearly,
\formula{
 \int_0^\infty \bigl((f'(s))^2 + \lambda A(s) (f(s))^2\bigl) ds & = \int_0^\infty \bigl((\ph_\lambda'(s))^2 + \lambda A(s) (\ph_\lambda(s))^2\bigl) ds \\
 & \hspace*{-10em} + 2 \int_0^\infty \bigl(\ph_\lambda'(s) g'(s) + \lambda A(s) \ph_\lambda(s) g(s)\bigl) ds + \int_0^\infty \bigl((g'(s))^2 + \lambda A(s) (g(s))^2\bigl) ds .
}
We already know that the first integral in the right-hand side is equal to $\psi(\lambda)$, while the middle one is zero. The last one is nonnegative, and it is equal to zero if and only if $g'(s) = 0$ for almost all $s$, which implies that $g$ is identically zero. The proof of~\eqref{eq:ph:min} is complete.

\subsection{Change of variable}
\label{sec:krein:change}

Suppose that $a(t)$ is a Borel function on $(0, \infty)$, with values in $[0, \infty]$, such that $\sigma(T) = \int_0^T (a(t))^{-1} dt$ is strictly increasing and finite for all $T > 0$, and such that $\sigma(t)$ diverges to infinity as $t \to \infty$. Then $A(\sigma(t)) = (a(t))^2$ defines a Borel function $A(s)$ on $(0, \infty)$.

Conversely, if a positive, Borel function $A(s)$ is given, then the corresponding $\sigma(t)$ and $a(t) = 1 / \sigma'(t)$ can be found by solving the ordinary differential equation $A(\sigma(t)) = (\sigma'(t))^{-2}$, that is, they are described by the identity
\formula{
 \int_0^{\sigma(t)} \sqrt{A(s)} \, ds & = t
}
for $t \in (0, \infty)$. In order that $\sigma(t)$ is indeed well-defined and continuous, we need to assume that the integral of $(A(s))^{1/2}$ is strictly increasing, finite and divergent to infinity as $s \to \infty$; for absolute continuity of $\sigma$ (required for the definition of $a(t) = (\sigma'(t))^{-1}$), additional conditions on $A(s)$ need to be imposed.

After a change of variable $s = \sigma(t)$, we find that
\formula{
 \int A(s) ds & = \int A(\sigma(t)) \sigma'(t) dt = \int a(t) dt .
}
Therefore, local integrability of $a(t)$ on $[0, \infty)$ is equivalent to local integrability of $A(s)$ on $[0, \infty)$.

Suppose that $\tilde{f}(t) = f(\sigma(t))$. Note that $\sigma(t)$ is absolutely continuous and monotone. It follows that if $f$ is absolutely continuous, then so is $\tilde{f}$. The converse is also true, since the inverse function $\sigma^{-1}(s)$ is absolutely continuous (as a consequence of the fact that $\sigma'(t) = (a(t))^{-1}$ is positive almost everywhere due to local integrability of $a(t)$). The derivatives (either pointwise or weak) of $f$ and $\tilde{f}$ are related by the identity
\formula{
 a(t) \tilde{f}'(t) & = a(t) f'(\sigma(t)) \sigma'(t) = f'(\sigma(t))
}
for almost all $t$.

If $f'$ is absolutely continuous, then it follows that $a(t) \tilde{f}'(t)$ is absolutely continuous, and furthermore
\formula[eq:op:sub]{
\begin{aligned}
 (a(t))^{-1} (a(t) \tilde{f}'(t))' & = (a(t))^{-1} f''(\sigma(t)) \sigma'(t) \\
 & = (a(t))^{-2} f''(\sigma(t)) = (A(\sigma(t)))^{-1} f''(\sigma(t)) .
\end{aligned}
}
In other words, the operator $(a(t))^{-1} \partial_t (a(t) \partial_t)$ is equivalent to the operator $(A(s))^{-1} \partial_s^2$ after a change of variable $s = \sigma(t)$. This explains the identification of Dirichlet-to-Neumann operators given by~\eqref{eq:oph:s} and~\eqref{eq:op:s} on one hand, and by~\eqref{eq:oph:t} and~\eqref{eq:op:t} on the other.

We will need a similar correspondence in terms of quadratic forms. If $f(\sigma(t)) = \tilde{f}(t)$, we clearly have
\formula[eq:inner:sub]{
 \int_0^\infty a(t) (\tilde{f}(t))^2 dt & = \int_0^\infty (a(t))^2 (f(\sigma(t)))^2 \sigma'(t) dt = \int_0^\infty A(s) (f(s))^2 ds .
}
Furthermore, if $f$ or $\tilde{f}$ is weakly differentiable, then
\formula[eq:form:sub]{
 \int_0^\infty a(t) (\tilde{f}'(t))^2 dt & = \int_0^\infty (a(t))^{-1} f'(\sigma(t)))^2 dt = \int_0^\infty (f'(s))^2 ds .
}

%
%

\section{Extension technique}
\label{sec:h}

\subsection{Quadratic form in halfspace}

Throughout the entire section we assume that $A(s)$ is a nonnegative, locally integrable function on $[0, \infty)$. Recall that $\hh = (0, \infty) \times \R^d$.

\begin{definition}
\label{def:formh:s}
For a function $u$ on $\hh$, we define
\formula[eq:formh:s]{
 \form_H(u, u) & = \int_0^\infty \int_{\R^d} \bigl((\partial_s u(s, x))^2 + A(s) |\nabla_x u(s, x)|^2\bigr) dx ds .
}
The domain of this form, denoted $\dom(\form_H)$, is the set of all locally integrable functions $u$ on $\hh$ which satisfy the following conditions: $u(s, x)$ is weakly differentiable with respect to $s$, $(A(s))^{1/2} u(s, x)$ is weakly differentiable with respect to $x$, the integral in~\eqref{eq:formh:s} is finite, and
\formula[eq:formh:s:dom]{
 \int_{s_0}^{s_1} \int_{\R^d} A(s) |u(s, x)|^2 dx ds < \infty
}
whenever $0 < s_0 < s_1$.
\end{definition}

Note that if $A(s)$ is locally bounded below by a positive constant, then we can simply say that $u(s, x)$ is weakly differentiable with respect to both $s$ and $x$. However, for general $A$, the function $(A(s))^{-1/2}$ may fail to be integrable, and therefore weak differentiability of $(A(s))^{1/2} u(s, x)$ with respect to $x$ need not imply weak differentiability of $u(s, x)$ with respect to $x$. In this case, the second term under the integral in~\eqref{eq:formh:s} should in fact be understood as $|\nabla_x ((A(s))^{1/2} u(s, x))|^2$; for simplicity, however, we abuse the notation and we use the less formal expression $A(s) |\nabla_x u(s, x)|^2$.

Clearly, $\form_H$ defined above is the quadratic form of the operator
\formula{
 \op_H u(s, x) & = \partial_s^2 u(s, x) + A(s) \Delta_x u(s, x) .
}
As this operator will only be used in the weak sense, we do not need to specify the domain of $\op_H$.

Let $\fourier_x u(s, \cdot)$ denote the Fourier transform of $u(s, \cdot)$. By Plancherel's theorem, one easily finds that $u \in \dom(\form_H)$ if and only if $\fourier_x u(s, \xi)$ is weakly differentiable with respect to $s$,
\formula[eq:formh:s:f:dom]{
 \int_{s_0}^{s_1} \int_{\R^d} A(s) |\fourier_x u(s, \xi)|^2 dx ds < \infty
}
whenever $0 < s_0 < s_1$, and the integral in the right-hand side of the identity
\formula[eq:formh:s:f]{
 \form_H(u, u) & = \int_0^\infty \int_{\R^d} \bigl(|\partial_s \fourier_x u(s, \xi)|^2 + A(s) |\xi|^2 |\fourier_x u(s, \xi)|^2\bigr) d\xi ds
}
is finite; the above equality expresses $\form_H$ in terms of the Fourier transform.

If $u \in \dom(\form_H)$, then $\partial_s u(s, \cdot)$, as an $\leb^2(\R^d)$-valued function, is locally integrable on $[0, \infty)$ (in the sense of Bochner's integral). Indeed, by Schwarz inequality,
\formula[eq:norm:est]{
 \left(\int_{s_0}^{s_1} \|\partial_s u(s, \cdot)\|_2 ds\right)^2 & \le (s_1 - s_0) \int_{s_0}^{s_1} \|\partial_s u(s, \cdot)\|_2^2 ds \le (s_1 - s_0) \form_H(u, u)
}
whenever $0 \le s_0 < s_1$ (this is an analogue of~\eqref{eq:ph:parts}). Furthermore, one easily sees that for almost all $S$, the Bochner integral
\formula{
 \int_0^S \partial_s u(s, \cdot) ds
}
is equal to $u_0(\cdot) + u(S, \cdot)$ for some function (`constant') $u_0 \in \leb^2(\R^d)$. If we denote $u(0, \cdot) = u_0(\cdot)$, then, after modification on a set of zero Lebesgue measure, we may assume that in fact 
\formula[eq:u:int]{
 u(S, \cdot) & = u(0, \cdot) + \int_0^S \partial_s u(s, \cdot) ds ,
}
for all $S \in [0, \infty)$. Observe that~\eqref{eq:u:int} and~\eqref{eq:norm:est} imply that
\formula[eq:u:est]{
 \|u(s, \cdot)\|_2 & \le \|u(0, \cdot)\|_2 + (s \form_H(u, u))^{1/2} .
}
From now on we will always assume that~\eqref{eq:u:int} holds for all $S \in [0, \infty)$ whenever we consider $u \in \dom(\form_H)$.

\subsection{Boundary form}

The trace $\form(f, f)$ of the quadratic form $\form_H(u, u)$ on the boundary of $\hh$ is defined as the minimal value of $\form_H(u, u)$ among all $u \in \dom(\form_H)$ which satisfy the boundary condition $u(0, x) = f(x)$. It is not very difficult to see that the minimisers $u$ are harmonic functions with respect to $\op_H$, and these harmonic functions can be described in terms of the Fourier transform and the functions $\ph_\lambda$ introduced in Section~\ref{sec:krein}. A short calculation reveals that in fact
\formula[eq:form]{
 \form(f, f) & = \int_{\R^d} \psi(|\xi|^2) |f(\xi)|^2 d\xi ,
}
where $\psi(\lambda) = -\ph_\lambda'(0)$ is a complete Bernstein function described by Theorem~\ref{thm:krein}.

We take~\eqref{eq:form} as the definition, with the domain $\dom(\form)$ defined to be the space of all $f \in \leb^2(\R^d)$ for which the integral in the definition~\eqref{eq:form} of $\form(f, f)$ is finite. With this definition, we will prove that indeed $\form(f, f)$ is the trace of $\form_H(u, u)$ on the boundary. We also denote by $\op = \psi(-\Delta)$ the Fourier multiplier with symbol $\psi(|\xi|^2)$, that is, as in~\eqref{eq:op},
\formula{
 \fourier (\op f)(\xi) & = \psi(|\xi|^2) \fourier f(\xi) ,
}
with domain $\dom(\op)$ consisting of those $f \in \leb^2(\R^d)$ for which $\psi(|\xi|^2) \fourier f(\xi)$ is square integrable.

Let $\ph(\lambda, s) = \ph_\lambda(s)$ be the function defined in Theorem~\ref{thm:krein}. We introduce the harmonic extension operator.

\begin{definition}
\label{def:ext}
For $f \in \leb^2(\R^d)$ we define its harmonic extension $u = \ext(f)$ to $\hh$ by means of Fourier transform,
\formula{
 \fourier_x u(s, \xi) & = \ph(|\xi|^2, s) \fourier f(\xi) ,
}
where $\fourier_x u(s, \cdot)$ is the Fourier transform of $u(s, \cdot)$.
\end{definition}

Since $\ph(\lambda, s)$ is bounded by $1$, the harmonic extension is well-defined, and we have $\|u(s, \cdot)\|_2 \le \|f\|_2$. We begin by observing that the Dirichlet-to-Neumann operator applied to $\ext(f)$ coincides with $\op f$.

\begin{theorem}
\label{thm:oph:s}
Let $f \in \leb^2(\R^d)$ and $u = \ext(f)$. Then $f \in \dom(\op)$ if and only if the limit in the definition of $\partial_s u(0, \cdot)$ exists in $\leb^2(\R^d)$, and in this case
\formula{
 \op f & = -\partial_s u(0, \cdot) = -\lim_{s \to 0^+} \partial_s u(s, \cdot) .
}
\end{theorem}

\begin{proof}
Recall that $-\partial_s \ph(\lambda, s)$ is decreasing and equal to $\psi(\lambda)$ for $s = 0$. Thus, if $f \in \dom(\op)$, then the desired result follows by dominated convergence. On the other hand, if the limit in the definition of $\partial_s u(0, \cdot)$ exists, then $\psi(\xi^2) f(\xi)$ is square integrable by monotone convergence, and so $f \in \dom(\op)$.
\end{proof}

Our next two results state that $\form(f, f) = \form_H(u, u)$ if $u = \ext(f)$, and that $\ext(f)$ indeed minimises $\form(u, u)$ among all $u \in \dom(\form_H)$ such that $u(0, x) = f(x)$.

\begin{theorem}
\label{thm:forms:s}
Let $f \in \leb^2(\R^d)$ and $u = \ext(f)$. Then $u \in \dom(\form_H)$ if and only if $f \in \dom(\form)$, and in this case
\formula{
 \form(f, f) & = \form_H(u, u) .
}
\end{theorem}

\begin{proof}
Since $\|u(s, \cdot)\|_2 \le \|f\|_2$ and since $A(s)$ is locally integrable,
\formula{
 \int_{s_0}^{s_1} \int_{\R^d} A(s) |u(s, x)|^2 dx ds < \infty
}
whenever $0 < s_0 < s_1$. By equality in~\eqref{eq:ph:min} and Fubini,
\formula[eq:forms:fourier]{
\begin{aligned}
 & \int_0^\infty \int_{\R^d} \bigl(|\partial_s \fourier_x u(s, \xi)|^2 + A(s) |\xi|^2 |\fourier_x u(s, \xi)|^2\bigr) d\xi ds \\
 & \qquad = \int_0^\infty \int_{\R^d} \bigl((\partial_s \ph(|\xi|^2, s))^2 + A(s) |\xi|^2 (\ph(|\xi|^2, s))^2\bigr) |\fourier f(\xi)|^2 d\xi ds \\
 & \qquad = \int_{\R^d} \psi(|\xi|^2) |\fourier f(\xi)|^2 d\xi .
\end{aligned}
}
If $u \in \dom(\form_H)$, then the left-hand side is equal to $\form_H(u, u)$ (by~\eqref{eq:formh:s:f}), and therefore $f \in \dom(\form)$ and $\form(f, f) = \form_H(u, u)$. Conversely, if $f \in \dom(\form)$, then the right-hand side of the equality~\eqref{eq:forms:fourier} is finite. Therefore, $\fourier_x u(s, \xi)$ is weakly differentiable with respect to~$s$, and the expression in~\eqref{eq:formh:s:f} is finite. We conclude that $u \in \dom(\form_H)$, as desired.
\end{proof}

\begin{theorem}
\label{thm:min}
Let $v \in \dom(\form_H)$, $f(x) = v(0, x)$ and $u = \ext(f)$. Then $f \in \dom(\op)$ and
\formula{
 \form_H(v, v) & \ge \form_H(u, u) = \form(f, f) .
}
Moreover, the space $\dom(\form_H)$ is a direct sum of $\dom_0$ and $\dom_\harm$, where
\formula{
 \dom_0 & = \{ u \in \dom(\form_H) : u(0, x) = 0 \text{ for almost all } x \in \R^d \} , \\
 \dom_\harm & = \{ \ext(f) : f \in \dom(\form) \}
}
are orthogonal to each other with respect to $\form_H$. 
\end{theorem}

\begin{proof}
Let $v \in \dom(\form_H)$ and $f(x) = v(0, x)$. By~\eqref{eq:formh:s:f},
\formula{
 \form_H(v, v) & = \int_0^\infty \int_{\R^d} \bigl(|\partial_s \fourier_x v(s, \xi)|^2 + A(s) |\xi|^2 |\fourier_x v(s, \xi)|^2\bigr) d\xi ds
}
By the ACL characterisation, after a modification on the set of zero Lebesgue measure, for almost all $\xi \in \R^d$, the function $\fourier_x v(\cdot, \xi)$ is absolutely continuous on $[0, \infty)$, and the pointwise and weak definitions of $\partial_s \fourier_x v(s, \xi)$ coincide for almost all $s$. Denote this modification by $w(s, \xi)$. For those $\xi$ for which $w(\cdot, \xi)$ is absolutely continuous, $w(0, \xi) \ne 0$ and $w(\cdot, \xi)$ is square integrable on $(0, \infty)$, we have, by~\eqref{eq:ph:min},
\formula{
 \int_0^\infty \bigl(|\partial_s w(s, \xi)|^2 + A(s) |\xi|^2 |w(s, \xi)|^2\bigr) ds & \ge \psi(|\xi|^2) |w(0, \xi)|^2
}
(we applied~\eqref{eq:ph:min} to $f(s) = \re (w(s, \xi) / w(0, \xi))$ and $\lambda = |\xi|^2$). The above inequality is also trivially true when $w(0, \xi) = 0$. These two cases cover almost all $\xi \in \R^d$. Therefore, by Fubini,
\formula{
 \form_H(v, v) & \ge \int_{\R^d} \psi(|\xi|^2) |w(0, \xi)|^2 d\xi .
}
Observe that $w(s, \xi)$ converges to $w(0, \xi)$ as $s \to 0^+$ for almost all $\xi \in \R^d$. On the other hand, by~\eqref{eq:u:est} we know that $v(s, \cdot)$ converges to $f$ in $\leb^2(\R^d)$. However, $w(s, \cdot) = \fourier_x v(s, \cdot)$ for almost all $s$, and so a subsequence of $w(s, \cdot)$ converges to $\fourier f$ in $\leb^2(\R^d)$. Therefore, $w(0, \xi) = \fourier f(\xi)$ for almost all $\xi \in \R^d$, and we conclude that $f \in \dom(\op)$ and
\formula{
 \form_H(v, v) & \ge \form(f, f) .
}
The first part of the theorem follows now by Theorem~\ref{thm:forms:s}. Furthermore, if we denote $u = \ext(f)$, then $u \in \dom_\harm$ and $v - u \in \dom_0$, so that indeed $\dom(\form_H)$ is a sum of $\dom_\harm$ and $\dom_0$.

Orthogonality of $\dom_\harm$ and $\dom_0$ follows now by a standard argument: if $u \in \dom_\harm$, $v \in \dom_0$ and $\alpha \in \R$, then $\form_H(u + \alpha v, u + \alpha v) \ge \form_H(u, u)$, which reduces to
\formula{
 2 \alpha \form_H(u, v) + \alpha^2 \form_H(v, v) \ge 0 ;
}
thus, $\form_H(u, v) = 0$. (Here, of course, $\form_H(u, v)$ denotes the Hermitian form corresponding to the quadratic form $\form_H(u, u)$).
\end{proof}

We conclude this section with a result that explains the name \emph{harmonic extension} used for the function $\ext(f)$.

\begin{theorem}
\label{thm:harm}
If $f \in \leb^2(\R^d)$, then the harmonic extension $u = \ext(f)$ satisfies
\formula[eq:harm:s]{
 \partial_s^2 u(s, x) + A(s) \Delta_x u(s, x) & = 0
}
in $\hh$ in the weak sense. Conversely, if $u$ satisfies~\eqref{eq:harm:s} in $\hh$ in the weak sense and $\|u(s, \cdot)\|_2$ is a bounded function of $s$, then $u = \ext(f)$ for some $f \in \leb^2(\R^d)$.
\end{theorem}

\begin{proof}
We understand~\eqref{eq:harm:s} as
\formula[eq:harm:w]{
 \int_0^\infty \int_{\R^d} u(s, x) (g''(s) h(x) + A(s) g(s) \Delta h(x)) dx ds & = 0
}
for all test functions $g \in \cont_c^\infty((0, \infty))$ and $h \in \cont_c^\infty(\R^d)$. If $u = \ext(f)$, by Plancherel's theorem the left-hand side of~\eqref{eq:harm:w} is equal to
\formula{
 \int_0^\infty \int_{\R^d} \fourier f(\xi) \ph(|\xi|^2, s) (g''(s) - A(s) g(s) |\xi|^2) \fourier h(\xi) d\xi ds .
}
For any $\xi \in \R^d$ the integral over $s \in (0, \infty)$ is zero due to the fact that $\ph_\lambda''(s) = -\lambda A(s) \ph_\lambda(s)$ in the weak sense. The first statement is thus proved.

To prove the second one, we use a similar argument: Plancherel's theorem implies that
\formula{
 \int_0^\infty \int_{\R^d} \fourier_x u(s, \xi) (g''(s) - A(s) g(s) |\xi|^2) \fourier h(\xi) d\xi ds & = 0 .
}
By considering a countable and linearly dense set of pairs $g \in \cont_c^\infty((0, \infty))$ and $h \in C_c^\infty(\R^d)$, we see that for almost all $\xi \in \R^d$, $\fourier_x u(\cdot, \xi)$ is a weak solution of $\ph''(s) = |\xi|^2 A(s) \ph(s)$. Such a solution is either a multiple of $\ph(|\xi|^2, s)$, or a function that diverges to $\pm \infty$ as $s \to \infty$. Since the $\leb^2(\R^d)$ norms of $\fourier_x u(s, \cdot)$ are bounded as $s \to \infty$, for almost all $\xi \in \R^d$ we have $\fourier_x u(s, \xi) = \ph(|\xi|^2, s) F(\xi)$ for some $F(\xi)$. Using again boundedness of $\|\fourier_x u(s, \cdot)\|_2$ we conclude that $F \in \leb^2(\R^d)$, and therefore $u = \ext(\fourier^{-1} F)$.
\end{proof}

\subsection{Change of variable}
\label{sec:h:change}

We now rephrase the results of the previous section in terms of the operator~\eqref{eq:oph:t}. Suppose that $a(t)$ is a locally integrable nonnegative function on $[0, \infty)$ such that $(a(t))^{-1}$ is also locally integrable, but not integrable, on $[0, \infty)$. An extension to the case when $(a(t))^{-1}$ is integrable on $[0, \infty)$ is discussed in Appendix~\ref{app:finite}.

Following Section~\ref{sec:krein:change}, we define
\formula{
 \sigma(T) & = \int_0^T (a(t))^{-1} dt .
}
Using the results of Section~\ref{sec:krein:change}, we can identify the quadratic form $\form_H$ defined earlier in this section with the following one.

\begin{definition}
\label{def:formh:t}
For a function $\tilde{u}$ on $\hh$, the quadratic form $\tilde{\form}_H(\tilde{u}, \tilde{u})$ is defined by the expression
\formula[eq:formh:t]{
 \tilde{\form}_H(\tilde{u}, \tilde{u}) & = \int_0^\infty \int_{\R^d} a(t) |\nabla_{t,x} \tilde{u}(t, x)|^2 dx dt .
}
The domain $\dom(\tilde{\form}_H)$ of this form is the set of all locally integrable functions $\tilde{u}$ on $\hh$ which satisfy the following conditions: $\tilde{u}(t, x)$ is weakly differentiable with respect to~$t$, $(a(t))^{1/2} \tilde{u}(t, x)$ is weakly differentiable with respect to $x$, the integral in~\eqref{eq:formh:t} is finite, and
\formula[eq:formh:t:dom]{
 \int_{t_0}^{t_1} \int_{\R^d} a(t) |\tilde{u}(t, x)|^2 dx dt < \infty
}
whenever $0 < t_0 < t_1$. Finally, if $f \in \leb^2(\R^d)$, then we define its harmonic extension $\tilde{u} = \extt{f}$ by $\tilde{u}(t, x) = u(\sigma(t), x)$, where $u = \ext(f)$.
\end{definition}

Recall that under our assumptions, the function $A(s)$ defined by $A(\sigma(t)) = (a(t))^2$ is locally integrable on $[0, \infty)$. Equivalence of the quadratic forms $\form_H$ and $\tilde{\form}_H$, as well as the corresponding Dirichlet-to-Neumann operators, is very simple when the coefficients are regular enough. In the general case, one needs to pay extra attention to domains.

\begin{lemma}
\label{lem:formh:sub}
Suppose that $u(\sigma(t), x) = \tilde{u}(t, x)$. Then $u \in \dom(\form_H)$ if and only if $\tilde{u} \in \dom(\tilde{\form}_H)$, and in this case $\form_H(u, u) = \tilde{\form}_H(\tilde{u}, \tilde{u})$.
\end{lemma}

\begin{proof}
Recall that if $f(\sigma(t)) = \tilde{f}(t)$, then absolute continuity of $f(s)$ is equivalent to absolute continuity of $\tilde{f}(t)$. By the ACL characterisation of weak differentiability, weak differentiability of $u(s, x)$ with respect to $s$ is equivalent to weak differentiability of $\tilde{u}(t, x)$ with respect to $t$.

Using the same method together with formula~\eqref{eq:inner:sub}, we see that weak differentiability of $(A(s))^{1/2} u(s, x)$ with respect to $x$ is equivalent to weak differentiability of $(a(t))^{1/2} \tilde{u}(t, x)$ with respect to $t$.

By~\eqref{eq:inner:sub} and~\eqref{eq:form:sub}, the integrals defining $\form_H(u, u)$ and $\tilde{\form}_H(\tilde{u}, \tilde{u})$ are equal. Similarly, by~\eqref{eq:inner:sub}, the integrals in~\eqref{eq:formh:s:dom} and~\eqref{eq:formh:t:dom} are equal if $s_0 = \sigma(t_0)$ and $s_1 = \sigma(t_1)$. Therefore, $u \in \dom(\form_H)$ if and only if $\tilde{u} \in \dom(\tilde{\form}_H)$, and we already noted that in this case $\form_H(u, u) = \tilde{\form}_H(\tilde{u}, \tilde{u})$.
\end{proof}

The following result follows almost immediately from Theorems~\ref{thm:oph:s} and~\ref{thm:harm}.

\begin{theorem}
\label{thm:oph:t}
Suppose that $f \in \leb^2(\R^d)$ and $\tilde{u} = \extt(f)$. Then $\tilde{u}$ satisfies
\formula[eq:harm:t]{
 \partial_{t,x}(a(t) \partial_{t,x} \tilde{u}(t, x)) & = 0
}
in $\hh$ in the weak sense. Furthermore, $f \in \dom(\op)$ if and only if any of the limits in the identity
\formula{
 \op f & = -\lim_{t \to 0^+} \frac{\tilde{u}(t, \cdot) - \tilde{u}(0, \cdot)}{\sigma(t)} = -\lim_{t \to 0^+} a(t) \partial_t \tilde{u}(t, \cdot)
}
exists in $\leb^2(\R^d)$. Finally, if $\tilde{u}$ satisfies~\eqref{eq:harm:t} in the weak sense and $\|u(t, \cdot)\|_2$ is a bounded function of $t$, then $\tilde{u} = \extt(f)$ for some $f \in \leb^2(\R^d)$.
\end{theorem}

\begin{proof}
The first and the last statements are merely a reformulation of Theorem~\ref{thm:harm}, combined with the identification~\eqref{eq:op:sub} of the operators $(A(s))^{-1} \partial_s^2$ and $(a(t))^{-1} \partial_t(a(t) \partial_t)$: for any test function $v \in C_c^\infty(\hh)$ and $\tilde{v}(t, x) = v(\sigma(t), x)$ we have
\formula{
 & \int_0^\infty \int_{\R^d} \partial_{t,x}(a(t) \partial_{t,x} \tilde{u}(t, x)) \tilde{v}(t, x) dx dt \\
\displaybreak[0]
 & \hspace*{5em} = \int_0^\infty \int_{\R^d} \tilde{u}(t, x) \partial_{t,x}(a(t) \partial_{t,x} \tilde{v}(t, x)) dx dt \\
\displaybreak[0]
 & \hspace*{5em} = \int_0^\infty \int_{\R^d} a(t) \tilde{u}(t, x) ((a(t))^{-1} \partial_t(a(t) \partial_t \tilde{v}(t, x)) + \Delta_x \tilde{v}(t, x)) dx dt \\
 & \hspace*{5em} = \int_0^\infty \int_{\R^d} A(s) u(s, x) ((A(s))^{-1} \partial_s^2 v(s, x) + \Delta_x v(s, x)) dx ds \\
 & \hspace*{5em} = \int_0^\infty \int_{\R^d} (\partial_s^2 u(s, x) + A(s) \Delta_x u(s, x)) v(s, x) dx ds .
}
The middle statement is a direct consequence of Theorem~\ref{thm:oph:s}. 
\end{proof}

By combining Lemma~\ref{lem:formh:sub} with Theorems~\ref{thm:forms:s} and~\ref{thm:min}, we immediately get the following result.

\begin{theorem}
\label{thm:forms:t}
Let $f \in \leb^2(\R^d)$ and $\tilde{u} = \extt(f)$. Then $f \in \dom(\form)$ if and only if $\tilde{u} \in \dom(\tilde{\form}_H)$. If $\tilde{v} \in \dom(\tilde{\form}_H)$, $f(x) = \tilde{v}(0, x)$ and $\tilde{u} = \extt(f)$, then $f \in \dom(\form)$ and
\formula{
 \tilde{\form}_H(\tilde{v}, \tilde{v}) & \ge \tilde{\form}_H(\tilde{u}, \tilde{u}) = \form(f, f) .
}
Moreover, the space $\dom(\tilde{\form}_H)$ is a direct sum of $\tilde{\dom}_0$ and $\tilde{\dom}_\harm$, where
\formula{
 \tilde{\dom}_0 & = \{ \tilde{u} \in \dom(\tilde{\form}_H) : \tilde{u}(0, x) = 0 \text{ for almost all } x \in \R^d \} , \\
 \tilde{\dom}_\harm & = \{ \extt(f) : f \in \dom(\form) \}
}
are orthogonal to each other with respect to $\tilde{\form}_H$.
\end{theorem}

%
%

\section{Examples}
\label{sec:ex}

Before we proceed with applications of the extension technique introduced in the previous section, we discuss several examples. Noteworthy, there is only a handful of known pairs of explicit coefficients $A(s)$ (or $a(t)$) and corresponding complete Bernstein functions $\psi(\lambda)$; Chapter~15 in~\cite{bib:ssv12} contains a concise table in a different (probabilistic) language. This is, however, not an essential problem in most applications of the variational principles of Theorem~\ref{thm:var:h}, because in order to use them, one typically does not require an explicit form of the coefficients: it is sufficient to know that appropriate coefficients $A(s)$ or $a(t)$ exist.

Throughout this section, as it was the case in the introduction, we drop tilde from the notation, and write $u(t, x)$ for what was denoted by $\tilde{u}(t, x)$ in Section~\ref{sec:h:change}. We also simply write $\ph_\lambda(t)$ instead of more formal $\ph_\lambda(\sigma(t))$.

Most examples below are arranged in the following way: we begin with coefficients $A(s)$ or $a(t)$ and find the corresponding $\ph_\lambda(s)$ or $\ph_\lambda(t)$. The harmonic extension $u = \ext(f)$ for the operator~\eqref{eq:oph:t} is then given by
\formula{
 \fourier_x u(t, \xi) & = \ph_\lambda(t) \fourier f(\xi) ,
}
where $\lambda = |\xi|^2$.

\subsection{Classical Dirichlet-to-Neumann operator}
\label{sec:ex:classical}

Let $A(s) = 1$, or $a(t) = 1$. Then $\sigma(t) = t$, so the two parametrisations are identical. Clearly, $\ph_\lambda(t) = \exp(-\lambda^{1/2} t)$, $\psi(\lambda) = \lambda^{1/2}$, and we recover the classical result: if $u(t, x)$ is harmonic in $\hh$ (that is, $\Delta_{t,x} u(t, x) = 0$ in $\hh$) with boundary value $u(0, x) = f(x)$, then
\formula[eq:classical]{
\begin{aligned}
 \int_{\R^d} |\xi| |\fourier f(\xi)|^2 d\xi & = \int_0^\infty \int_{\R^d} |\nabla_{t,x} u(t, x)|^2 dx dt , \\
 (-\Delta)^{1/2} f & = \partial_t u(0, \cdot) .
\end{aligned}
}
Two things need to be clarified here. In this section we understand that in expressions similar to~\eqref{eq:classical} one side is defined (i.e.\@ the integral is finite in either side of the first equality; $|\xi| \fourier f(|\xi|)$ is square integrable in the left-hand side of the second equality; the partial derivative exists with a limit in $\leb^2(\R^d)$ in the right-hand side of the second equality) if and only if the other one is also defined. Furthermore, we need to impose some boundedness condition to assert that indeed $u = \ext(f)$ in the sense of Definition~\ref{def:ext}. By Theorem~\ref{thm:harm}, it is sufficient to assume that $u$ is twice differentiable in the weak sense, with $u(s, \cdot)$ bounded in $\leb^2$ for $s \in (0, \infty)$, and that harmonicity is understood in the weak sense. Note, however, that in most cases this condition can be significantly relaxed.

\subsection{Caffarelli--Silvestre extension technique}
\label{sec:ex:cs}

Let $\alpha \in (0, 2)$, and define two constants:
\formula{
 c_\alpha & = 2^\alpha \Gamma(\alpha/2) |\Gamma(-\alpha/2)|^{-1} , & C_\alpha & = 2^{1 - \alpha/2} (\Gamma(\alpha/2))^{-1} .
}
Consider the coefficients
\formula{
 A(s) & = \alpha^{-2} c_\alpha^{2/\alpha} s^{2/\alpha - 2} && \text{or} & a(t) & = \alpha^{-1} c_\alpha t^{1 - \alpha} .
}
Then one finds that
\formula{
 \ph_\lambda(s) & = C_\alpha (c_\alpha \lambda^{\alpha/2} s)^{1/2} K_{\alpha/2}((c_\alpha \lambda^{\alpha/2} s)^{1/\alpha}) , \\
 \psi(\lambda) & = \lambda^{\alpha/2} ,
}
where $K_{\alpha/2}$ is the modified Bessel function of the second kind. In variable $t$, we have
\formula{
 \sigma(t) & = c_\alpha^{-1} t^\alpha , \\
 \ph_\lambda(t) & = C_\alpha (\lambda^{1/2} t)^{\alpha/2} K_{\alpha/2}(\lambda^{1/2} t) , \\
 \psi(\lambda) & = \lambda^{\alpha/2} ,
}
Therefore, if $u(t, x)$ satisfies
\formula{
 \nabla_{t,x} (t^{1 - \alpha} \nabla_{t,x} u(t, x)) = 0
}
in $\hh$ with boundary value $u(0, x) = f(x)$, then
\formula{
 \int_{\R^d} |\xi|^\alpha |\fourier f(\xi)|^2 d\xi & = \int_0^\infty \int_{\R^d} \alpha^{-1} c_\alpha t^{1 - \alpha} |\nabla_{t,x} u(s, x)|^2 dx ds
}
and
\formula{
 (-\Delta)^{\alpha/2} f & = -\alpha^{-1} c_\alpha \lim_{t \to 0^+} t^{1 - \alpha} \partial_t u(t, \cdot) .
}
This can be re-written in variable $s$: if $u(s, x)$ satisfies
\formula{
 \partial_s^2 u(s, x) + \alpha^{-2} c_\alpha^{2/\alpha} s^{2/\alpha - 2} \Delta_x u(s, x) = 0
}
in $\hh$ with boundary value $u(0, x) = f(x)$, then
\formula{
 \int_{\R^d} |\xi|^\alpha |\fourier f(\xi)|^2 d\xi & = \int_0^\infty \int_{\R^d} \bigl( (\partial_s u(s, x))^2 + \alpha^{-2} c_\alpha^{2/\alpha} s^{2 - 2/\alpha} \nabla_x u(s, x)|^2 \bigr) dx ds
}
and
\formula{
 (-\Delta)^{\alpha/2} f & = -\partial_s u(0, \cdot) .
}
This example was first studied by Molchanov and Ostrovski~\cite{bib:mo69} in the language of stochastic processes, and then recently by Caffarelli and Silvestre~\cite{bib:cs07} in the context of non-local partial differential equations; see Introduction for further references.

\subsection{Quasi-relativistic operator}
\label{sec:ex:relat}

Let $m > 0$ and define
\formula{
 A(s) & = (1 + 2 m s)^{-2} && \text{or} & a(t) & = e^{-2 m t} .
}
Then
\formula{
 \ph_\lambda(s) & = (1 + 2 m s)^{(2m)^{-1} (m - (m^2 + \lambda)^{1/2})} , \\
 \psi(\lambda) & = (m^2 + \lambda)^{1/2} - m .
}
In variable $t$, we have
\formula{
 \sigma(t) & = (2 m)^{-1} (e^{2 m t} - 1) , \\
 \ph_\lambda(t) & = e^{(m - (m^2 + \lambda)^{1/2}) t} .
}
Therefore, the quasi-relativistic operator $(-\Delta + m^2)^{1/2} - m$ can be expressed as the Dirichlet-to-Neumann operator for differential operators
\formula{
 \partial_s^2 u(s, x) + (1 + 2 m s)^{-2} \Delta_x u(s, x)
}
or
\formula{
 \nabla_{t,x} (e^{-2 m t} \nabla_{t,x} u(t, x))
}
in $\hh$. This example was studied, in the probabilistic context, in~\cite{bib:py03}.

\subsection{A string of finite length}
\label{sec:ex:relat:dual}

A closely related example is obtained by considering
\formula{
 A(s) & = (1 - 2 m s)^{-2} && \text{or} & a(t) & = e^{2 m t} ,
}
where in this case the range of $s$ is finite: $s \in (0, (2 m)^{-1})$. The corresponding version of the results of Section~\ref{sec:h} is discussed in Appendix~\ref{app:finite}.

Note that $((2 m)^{-1} - s) A(s)$ is not integrable, so that $\ph_\lambda(s)$ automatically satisfies the Dirichlet condition at $s = (2 m)^{-1}$. In particular, $A(s)$ is not integrable, and therefore $u(s, \cdot)$ automatically converges in $\leb^2(\R^d)$ to zero as $s \to ((2 m)^{-1})^-$ for every $u \in \dom(\form_H)$. We have
\formula{
 \ph_\lambda(s) & = (1 - 2 m s)^{(2m)^{-1} (m + (m^2 + \lambda)^{1/2})} , \\
 \psi(\lambda) & = (m^2 + \lambda)^{1/2} + m ,
}
and in variable $t$,
\formula{
 \sigma(t) & = (2 m)^{-1} (1 - e^{-2 m t}) , \\
 \ph_\lambda(t) & = e^{-(m + (m^2 + \lambda)^{1/2}) t} .
}
This gives a harmonic extension problem for the operator $(-\Delta + m^2)^{1/2} + m$, which differs by a constant from the quasi-relativistic operator from the previous example.

Since $(\lambda + m^2)^{1/2}$ is complete Bernstein function, there must be a corresponding coefficient $A(s)$ or $a(t)$. To our knowledge, the explicit form for this coefficient is not known. There is, however, a different way to represent $(-\Delta + m^2)^{1/2}$: consider the classical extension problem for $-\Delta + m^2$ instead of $-\Delta$. This approach was exploited in, for example,~\cite{bib:a16,bib:bmr10}.

\subsection{Quasi-relativistic-type operators}
\label{sec:ex:relat:general}

More generally, if $\psi(-\Delta)$ is the Dirichlet-to-Neumann operator for the differential operator~\eqref{eq:oph:t}, then it is easy to construct an analogous representation for $\psi(\mu - \Delta) - \psi(\mu)$, where $\mu > 0$. Indeed, denote
\formula{
 \op_a f(t) & = (a(t))^{-1} (a(t) f'(t))' = f''(t) + (a(t))^{-1} a'(t) f'(t)
}
and observe that
\formula{
 \op_a(\ph_\mu f) & = \ph_\mu \op_a f + 2 \ph_\mu' f' + f \op_a \ph_\mu \\
 & = \ph_\mu \op_a f + 2 \ph_\mu' f' + \mu f \ph_\mu .
}
On the other hand, if $b(t) = a(t) (\ph_\mu(t))^2$, then
\formula{
 \op_b f(t) & = f''(t) + (a(t))^{-1} a'(t) + 2 (\ph_\mu(t))^{-1} \ph_\mu'(t) f'(t) .
}
Therefore,
\formula{
 \ph_\mu^{-1} \op_a(\ph_\mu f) & = \op_b f + \mu f .
}
Set $f_\lambda(t) = (\ph_\mu(t))^{-1} \ph_{\mu + \lambda}(t)$. Then 
\formula{
 \op_b f_\lambda & = \ph_\mu^{-1} \op_a(\ph_\mu f_\lambda) - \mu f_\lambda = \ph_\mu^{-1} \op_a \ph_\lambda - \mu f_\lambda = (\mu + \lambda) \ph_\mu^{-1} \ph_\lambda - \mu f_\lambda = \lambda f_\lambda .
}
Furthermore,
\formula{
 \lim_{t \to 0^+} b(t) f_\lambda'(t) & = \lim_{t \to 0^+} a(t) (\ph_\mu(t))^2 ((\ph_\mu(t))^{-1} \ph_{\mu + \lambda}(t))' \\
 & = \lim_{t \to 0^+} a(t) (\ph_{\mu + \lambda}'(t) \ph_\mu(t) - \ph_{\mu + \lambda}(t) \ph_\mu'(t)) = \psi(\mu + \lambda) - \psi(\mu) .
}
We conclude that the coefficient $b(t)$ indeed corresponds to $\psi(\mu + \lambda) - \psi(\mu)$.

In particular, the operator $(m^2 - \Delta)^{\alpha/2} - m^\alpha$ is the Dirichlet-to-Neumann operator for the differential operator
\formula{
 \alpha^{-1} c_\alpha m^\alpha \nabla_{t,x} (t (K_{\alpha/2}(m t))^2 \nabla_{t,x} u(t, x)) ;
}
here we set $\mu = m^2$. This calculation is due to~\cite{bib:dy06} in probabilistic context.

Note that the argument used in this section works well in variable $t$, but it is not easy to reproduce in variable $s$: the operator $(A(s))^{-1} (\ph_\mu(s) f(s))''$ is not of the form $(B(s))^{-1} f(s) + \mu f(s)$ unless another change of variable is introduced.

\subsection{Operators in the theory of linear water waves}
\label{sec:ex:waves}

In the theory of linear water waves, one often considers the Dirichlet-to-Neumann operator for $A(s) = 1$, or $a(t) = 1$, in a finite interval $(0, R)$ (which represents the depth of the ocean). Since $\sigma(t) = t$, the two parametrisations coincide. A version of the results of Section~\ref{sec:h} adapted to the present setting is discussed in Appendix~\ref{app:finite}.

Imposing Neumann boundary condition at $t = R$ is equivalent to setting $A(s) = 0$ for $s \ge R$, and one easily finds that
\formula{
 \ph_\lambda(t) & = \frac{\cosh(\lambda^{1/2} (R - t))}{\cosh (\lambda^{1/2} R)} \, , \\
 \psi(\lambda) & = \lambda^{1/2} \tanh(\lambda^{1/2} R) .
}
The corresponding operator has Fourier symbol $|\xi| \tanh(R |\xi|)$.

Dirichlet condition at $t = R$ is equivalent to considering a string of finite length $R$, and we get
\formula{
 \ph_\lambda(t) & = \frac{\sinh(\lambda^{1/2} (R - t))}{\sinh (\lambda^{1/2} R)} \, , \\
 \psi(\lambda) & = \lambda^{1/2} (\tanh(\lambda^{1/2} R))^{-1} .
}
Therefore, the Dirichlet-to-Neumann operator has Fourier symbol $|\xi| (\tanh(R |\xi|))^{-1}$.

\subsection{Complementary operators}
\label{sec:ex:complementary}

The previous examples suggest that if the coefficient $a(t)$ corresponds to the operator $\psi(-\Delta)$, then the coefficient $b(t) = (a(t))^{-1}$ corresponds to the complementary operator $(-\Delta) (\psi(-\Delta))^{-1}$; if problems in a finite strip $t \in (0, R)$ are considered, Dirichlet and Neumann boundary conditions at $t = R$ need to be exchanged. This is indeed a case, as we will briefly show.

Let $\ph_\lambda$ be the function described by Theorem~\ref{thm:krein} in variable $t$; that is, $\ph_\lambda$ is a non-increasing non-negative function such that $\ph_\lambda(0) = 1$ and $(a(t))^{-1} (a(t) \ph_\lambda'(t))' = \lambda \ph_\lambda(t)$. We also know that $a(t) \ph_\lambda'(t)$ converges to $-\psi(\lambda)$ as $t \to 0^+$.

Let $\thet_\lambda(t) = -(\psi(\lambda))^{-1} a(t) \ph_\lambda'(t)$. Then $\thet_\lambda(0) = 1$, $\thet_\lambda(t) \ge 0$, and
\formula{
 \thet_\lambda'(t) & = -(\psi(\lambda))^{-1} (a(t) \ph_\lambda'(t))' = -\lambda (\psi(\lambda))^{-1} a(t) \ph_\lambda(t) .
}
In particular, $\thet_\lambda$ is non-increasing. Furthermore,
\formula{
 a(t) ((a(t))^{-1} \thet_\lambda'(t))' & = -(\psi(\lambda))^{-1} a(t) \ph_\lambda'(t) = \lambda \thet_\lambda(t) .
}
Therefore, $\thet_\lambda$ is the analogue of the function $\ph_\lambda$ for the coefficient $b(t) = (a(t))^{-1}$ (instead of $a(t)$). Finally,
\formula{
 (a(t))^{-1} \thet_\lambda'(t) & = -\lambda (\psi(\lambda))^{-1} \ph_\lambda(t) = \lambda (\psi(\lambda))^{-1} \thet_\lambda(t) .
}
converges to $\lambda (\psi(\lambda))^{-1}$ as $t \to 0^+$. This completes the proof our claim.

Noteworthy, the operators~\eqref{eq:oph:t} corresponding to $a(t)$ and $b(t) = (a(t))^{-1}$ are
\formula{
 & \partial_t^2 + \frac{a'(t)}{a(t)} \, \partial_t && \text{and} && \partial_t^2 - \frac{a'(t)}{a(t)} \, \partial_t ;
}
in particular, they are (formally) adjoint one to the other.

Let us state the above property in terms of variable $s$. Recall that the coefficient $a(t)$ corresponds to $A(s)$ such that $A(\sigma(t)) = (a(t))^2$. In a similar way, $b(t) = (a(t))^{-1}$ corresponds to $B(s)$ such that $B(\tau(t)) = (b(t))^2$, where $\tau(T) = \int_0^T (b(t))^{-1} dt$. It follows that
\formula{
 \int_0^{\sigma(T)} A(s) ds & = \int_0^T A(\sigma(t)) \sigma'(t) dt = \int_0^T a(t) dt = \int_0^T (b(t))^{-1} dt = \tau(T) ,
}
and similarly
\formula{
 \int_0^{\tau(T)} B(s) ds & = \sigma(T) .
}
In other words, $A([0, \sigma(T))) = \tau(T)$ and $B([0, \tau(T))) = \sigma(T)$, that is, the distribution functions $A([0, s))$ and $B([0, s))$ form a pair of inverse functions.

The above observation is a special case of a general fact in Krein's spectral theory of strings: the distribution functions of two strings $A$ and $B$ form a pair of (generalised) inverse functions if and only if the complete Bernstein functions that correspond to $A$ and $B$ are $\psi(\lambda)$ and $\lambda (\psi(\lambda))^{-1}$; see~\cite{bib:dm76}.

\subsection{A non-standard example}
\label{sec:ex:nonstandard}

An interesting class of examples is obtained by considering $\alpha \in \R$ and
\formula{
 a(t) & = (1 - t)^{1 - 2 \alpha} ,
}
where the range of $t$ is finite: $t \in (0, 1)$. Again we refer to Appendix~\ref{app:finite} for the discussion of the results of Section~\ref{sec:h} within the present context.

When $\alpha > 0$, then in addition a Dirichlet boundary condition needs to be imposed on $\ph_\lambda(t)$ at $t = 1$. It is easy to verify that
\formula{
 \ph_\lambda(t) & = \frac{(1 - t)^\alpha I_\alpha(\lambda^{1/2} (1 - t))}{I_\alpha(\lambda^{1/2})} \, , \\
 \psi(\lambda) & = \frac{\lambda^{1/2} I_{\alpha - 1}(\lambda^{1/2})}{I_\alpha(\lambda^{1/2})} \, ,
}
where $I_\alpha$ is the modified Bessel function of the first kind. This can be re-written in variable $s$ as follows. We find that if $\alpha \ne 0$, then
\formula{
 \sigma(t) & = \begin{cases} (2 \alpha)^{-1} (1 - (1 - t)^{2 \alpha}) & \text{if $\alpha \ne 0$,} \\ - \log(1 - t) & \text{if $\alpha = 0$,} \end{cases}
}
so that $s \in (0, \infty)$ if $\alpha \le 0$, while $s \in (0, (2 \alpha)^{-1})$ if $\alpha > 0$. It follows that
\formula{
 A(s) & = \begin{cases} (1 - 2 \alpha s)^{1/\alpha - 2} & \text{if $\alpha \ne 0$,} \\ e^{-2 s} & \text{if $\alpha = 0$.} \end{cases}
}
A formula for $\ph_\lambda(s)$ can be given, but it is of little use.

Observe that when $\alpha > 0$, then $((2 \alpha)^{-1} - s) A(s)$ is always integrable, and $A(s)$ is integrable if and only if $\alpha > 1$. Therefore, if $0 < \alpha < 1$, one needs to impose an additional condition that $u(s, \cdot)$ converges to $0$ in $\leb^2(\R^d)$ as $s \to ((2 \alpha)^{-1})^-$ in the definition of $\dom(\form_H)$.

%
%

\section{Variational principle}
\label{sec:var}

Let $V(x)$ be a locally bounded function on $\R^d$, bounded from below. We consider the Schr\"odinger operator $\op + V(x) = \psi(-\Delta) + V(x)$ and its quadratic form
\formula{
 \form_V(f, f) & = \form(f, f) + \int_{\R^d} V(x) (f(x))^2 dx \\
 & = \int_{\R^d} \psi(|\xi|^2) |\fourier f(\xi)|^2 d \xi + \int_{\R^d} V(x) (f(x))^2 dx ,
}
with domain $\dom(\form_V)$ equal to the set of those $f \in \dom(\form)$ for which $V(x) (f(x))^2$ is integrable. In order to use the harmonic extension technique, here we assume that $\psi$ is a complete Bernstein function.

Standard arguments show that if $\psi$ is unbounded and $V(x)$ is confining (that is, it converges to infinity as $|x| \to \infty$), then the operator $L + V(x)$ has discrete spectrum. In fact, a more general statement is true; we refer to~\cite{bib:lsw10} for further discussion.

Under the above assumptions, there is a complete orthonormal sequence of eigenfunctions $f_n$ of $\op + V(x)$, with corresponding eigenvalues $\mu_n$ arranged in a non-increasing way, and furthermore the sequence $\mu_n$ diverges to infinity.

Recall that $f$ is an eigenfunction with eigenvalue $\mu$ if and only if
\formula{
 \form_V(f, g) & = \mu \int_{\R^d} f(x) g(x) dx
}
for all $g \in \dom(\form_V)$. The eigenvalues below the essential spectrum are described by standard variational principles. For simplicity, we only state the results when $\psi$ and $V$ converge to infinity at infinity, and we denote by $f_1, f_2, \ldots$ the orthonormal sequence of eigenfunctions of $\op + V(x)$, with eigenvalues $\mu_1 \le \mu_2 \le \ldots$

\begin{theorem}
\label{thm:var}
The eigenvalues are given by the variational formula
\formula{
 \mu_n = \inf \{ \form_V(f, f) : {} & f \in \dom(\form_V), \|f\|_2 = 1, \\
 & \text{$f$ is orthogonal to $f_1, f_2, \ldots, f_{n - 1}$ in $\leb^2(\R^d)$} \} ,
}
and $f_n$ is one of the functions $f$ for which the infimum is attained. Another description is provided by the $\min$-$\max$ principle
\formula{
 \mu_n & = \inf \{ \sup \{ \form_V(f, f) : f \in \subdom, \|f\|_2 = 1 \} : \\
 & \qquad\qquad \text{$\subdom$ is an $n$-dimensional subspace of $\dom(\form_V)$} \} .
}
\end{theorem}

Let $\form_H$ be the quadratic form described in Definition~\ref{def:formh:s}, and let $\dom(\form_{H,V})$ be the space of those $u \in \dom(\form_H)$ for which $V(x) (u(0, x))^2$ is integrable. For $u \in \dom(\form_{H,V})$ define
\formula{
 \form_{H,V}(u, u) & = \form_H(u, u) + \int_{\R^d} V(x) (u(0, x))^2 dx .
}
In a similar way we define $\tilde{\form}_{H,V}$ and its domain $\dom(\tilde{\form}_{H,V})$ using the form given in Definition~\ref{def:formh:t}. By the identification of $\form$ with $\form_H$ or $\tilde{\form}_H$, we immediately have the following result.

\begin{theorem}
\label{thm:var:h}
The two variational principles of Theorem~\ref{thm:var} can be rewritten as
\formula[eq:var:h]{
\begin{aligned}
 \mu_n = \inf \{ \form_{H,V}(u, u) : {} & u \in \dom(\form_{H,V}), \|u(0, \cdot)\|_2 = 1, \\
 & \text{$u(0, \cdot)$ is orthogonal to $f_1, \ldots, f_{n - 1}$ in $\leb^2(\R^d)$} \} ,
\end{aligned}
}
and $u_n = \ext(f_n)$ is one of the functions $u$ for which the infimum is attained. We also have
\formula{
 \mu_n & = \inf \{ \sup \{ \form_{H,V}(u, u) : u \in \subdom, \|u(0, \cdot)\|_2 = 1 \} : \\
 & \hspace*{3em} \text{$\subdom$ is a subspace of $\dom(\form_{H,V})$ such that $\{u(0, \cdot) : u \in \subdom\}$ is $n$-dimensional} \} .
}
Similar expressions in terms of $\tilde{\form}_{H,V}$ are valid.
\end{theorem}

\begin{proof}
Let $\eta_n$ denote the infimum in the right-hand side of~\eqref{eq:var:h}. By considering $u = \ext(f_n)$ and observing that $\form_{H,V}(u, u) = \form_V(f_n, f_n) = \mu_n$, we immediately see that $\mu_n \ge \eta_n$. On the other hand, for any $u$ as in the definition of $\mu_n$, we have $\form_{H,V}(u, u) \ge \form_V(f, f) \ge \mu_n$, where $f = u(0, \cdot)$, so that $\eta_n \ge \mu_n$.

The proof of the second statement is very similar.
\end{proof}

Theorem~\ref{thm:var:h} turns out to be useful for two reasons. The form $\form_H$ and the corresponding operator $\op_H$ are local, and therefore geometrical properties of functions harmonic with respect to $\op_H$ are easier to study. This is illustrated by the Courant--Hilbert nodal line theorem in Section~\ref{sec:ch}. Furthermore, it is often much simpler to evaluate $\form_H(u, u)$ for appropriate function $u$ than to calculate (or estimate) $\form(f, f)$ for the corresponding boundary value $f(x) = u(0, x)$. Since $\form_H(u, u) \ge \form(f, f)$, this method can be used to find bounds for eigenvalues $\mu_n$, as indicated in Section~\ref{sec:eig}.

%
%

\section{Courant--Hilbert nodal line theorem}
\label{sec:ch}

Throughout this section we assume that $V$ is a confining potential, and $\psi$ is an unbounded complete Bernstein function. We consider the quadratic form $\form_H$ described in Definition~\ref{def:formh:s}, as well as
\formula{
 \form_{H,V}(u, u) & = \form_H(u, u) + \int_{\R^d} V(x) (u(0, x))^2 dx
}
introduced in Section~\ref{sec:var}. In a similar way we consider $\tilde{\form}_H$ and $\tilde{\form}_{H,V}$. 

Let $f_n$ be the sequence of eigenfunctions of the operator corresponding to $\form_{H,V}$, and let $\mu_n$ be the corresponding eigenvalues, arranged in a non-decreasing way. We define $u_n = \ext(f_n)$ to be the harmonic extension of $f_n$.

We will also assume that $\exp(-t \psi(|\xi|^2))$ is integrable over $\R^d$ for any $t > 0$. Under this assumption it is known that $f_n$ is continuous on $\R^d$, and $\fourier f_n$ is integrable (see Section~\ref{sec:fk} for further discussion). Using Fourier inversion formula and dominated convergence, one easily finds that in this case $u_n$ is continuous on $\hh$.

Fix $n$ and denote by $D$ a \emph{nodal domain} of $u_n$, that is, a connected component of $\{ (s, x) \in \hh : u_n(s, x) \ne 0 \}$. Furthermore, let $v(s, x) = \ind_D(s, x) u_n(s, x)$ be the corresponding \emph{nodal part} of $u_n$.

\begin{lemma}
\label{lem:nodal}
Any nodal part $v$ of $u_n = \ext(f_n)$ is weakly differentiable, with $\nabla_{s,x} v(s, x) = \ind_D(s, x) \nabla_{s,x} u_n(s, x)$. In particular, $v \in \dom(\form_H)$.
\end{lemma}

\begin{proof}
The result follows from the ACL characterisation of weak differentiability in a rather standard way. The details are, however, somewhat technical, and therefore we outline the proof.

After a modification on a set of zero Lebesgue measure, $u_n$ has the ACL property; since $u_n$ is already continuous, in fact no modification is needed. Fix a line on which $u_n$ is absolutely continuous. We will argue that from the definition of absolute continuity it follows that $v = \ind_D u_n$ is also absolutely continuous on this line.

Fix $\eps > 0$ and choose $\delta > 0$ according to the definition of absolute continuity of $u_n$ on the chosen line. Consider a collection of mutually disjoint intervals $(p_j, q_j)$ of total length less then $\delta$, and replace any interval $(p_j, q_j)$ which intersects $\partial D$ by two sub-intervals $(p_j, r_j)$ and $(r_j, q_j)$, for arbitrary $r_j \in (p_j, q_j) \cap \partial D$. For notational convenience, set $r_j = q_j$ if $(p_j, q_j)$ has no common point with $\partial D$. Then the sum of $|v(p_j) - v(q_j)|$ is easily shown not to exceed the sum of $|u_n(p_j) - u_n(r_j)| + |u_n(r_j) - u_n(q_j)|$. Thus, $v$ is absolutely continuous on the line considered.

Furthermore, if $\partial$ stands for the derivative along the chosen line, $\partial v(p) = \partial u_n(p)$ for $p \in D$ and $\partial v(p) = 0$ if $p \notin \overline{D}$. Suppose now that $p \in \partial D$. If there is a sequence of points $p_n \notin \overline{D}$ on the chosen line convergent to $p$, then $\partial v(p) = 0$. Finally, there is only a countable number of points $p \in \partial D$ on the chosen line for which such a sequence $p_n$ does not exist. Thus, $\partial v(p) = \partial u_n(p) \ind_D(p)$ for almost all $p$ on the chosen line.

It follows that $v$ has the ACL property, with $\nabla_{s,x} v(s, x) = \ind_D(s, x) \nabla_{s,x} u_n(s, x)$ almost everywhere. The desired result follows by the ACL characterisation of weak differentiability.
\end{proof}

Let $v$ be a nodal part of $u_n$, and let $g(x) = v(0, x)$. (Note that $g$ need not be a nodal part of $f_n$). Since $f_n$ is an eigenfunction, we have $\form_V(f_n, g) = \mu_n \int_{\R^d} f_n(x) g(x) dx$, that is,
\formula{
 \form_{H,V}(u_n, v) & = \mu_n \int_{\R^d} u_n(0, x) v(0, x) dx .
}
However, $u_n(s, x) v(s, x) = (v(s, x))^2$, and $\nabla_{s,x} u_n(s, x) \cdot \nabla_{s,x} v(s, x) = |\nabla_{s,x} v(s, x)|^2$ (the latter equality holds almost everywhere). Therefore,
\formula[eq:nodal:form]{
 \form_{H,V}(v, v) & = \form_{H,V}(u_n, v) = \mu_n \int_{\R^d} u_n(0, x) v(0, x) dx = \mu_n \int_{\R^d} (v(0, x))^2 dx .
}

The weak Courant--Hilbert theorem holds in full generality. For a strong version, we need the unique continuation property.

\begin{theorem}[a special case of Theorem~17.2.6 in~\cite{bib:h85}]
\label{thm:ucp}
If $A(s)$ is positive and locally Lipschitz continuous, $f \in \dom(\form)$ and $u = \ext(f)$, then the set $\{(s, x) \in \hh : u(s, x) = 0\}$ has either zero of full Lebesgue measure.
\end{theorem}

Note that $A(s)$ satisfies the assumptions of the above theorem if and only if the corresponding coefficient $a(t)$ is positive and locally Lipschitz continuous, for either condition implies that both $\sigma$ and the inverse function $\sigma^{-1}$ are locally $\cont^1$.

Noteworthy, there is no simple condition on the complete Bernstein function $\psi(\xi)$ which implies that the corresponding coefficient $A(s)$ is Lipschitz continuous. This condition is, however, satisfied for all examples in Section~\ref{sec:ex}.

The proof of the following Courant--Hilbert-type result is standard. For completeness, we provide the details.

\begin{theorem}
\label{thm:ch}
Suppose that $\exp(-t \psi(|\xi|^2))$ is integrable over $\R^d$ for any $t > 0$. The number of nodal parts of the harmonic extension $u_n = \ext(f_n)$ of the $n$-th eigenfunction does not exceed:
\begin{enumerate}[label={\rm (\alph*)}]
\item $\max \{j \in \N : \mu_j = \mu_n\}$ in the general case (weak version);
\item $\min \{j \in \N : \mu_j = \mu_n\}$ if $A(s)$ is positive and locally Lipschitz (strong version).
\end{enumerate}
\end{theorem}

\begin{proof}
Both statements are proved by a very similar argument: let $N = \min \{j \in \N : \mu_j > \mu_n\}$ for the weak version, $N = \min \{j \in \N : \mu_j = \mu_n\}$ for the strong version. Suppose that $u_n$ has at least $N$ nodal parts $v_1, v_2, \ldots, v_N$. Let $v = \sum_{j = 1}^N \alpha_j v_j$ be a non-zero linear combination of these nodal parts such that $v(0, \cdot)$ is orthogonal to $f_1, f_2, \ldots, f_{N - 1}$.

By Lemma~\ref{lem:nodal}, $v_j \in \dom(\form_{H,V})$, and by~\eqref{eq:nodal:form},
\formula{
 \form_{H,V}(v_j, v_j) & = \mu_n \int_{\R^d} (v_j(0, x))^2 dx .
}
Furthermore, if $i \ne j$, then $v_i(0, x) v_j(0, x)$ and $\nabla v_i(s, x) \cdot \nabla v_j(s, x)$ are equal to zero almost everywhere (again by Lemma~\ref{lem:nodal}), and therefore $\int_{\R^d} v_i(0, x) v_j(0, x) dx = 0$ and $\form_{H,V}(v_i, v_j) = 0$. It follows that $v \in \dom(\form_{H,V})$ and
\formula[eq:nodal:form:norm]{
 \form_{H,V}(v, v) & = \sum_{j = 1}^N \form_{H,V}(v_j, v_j) = \sum_{j = 1}^N \mu_n \int_{\R^d} (v_j(0, x))^2 dx = \mu_n \int_{\R^d} (v(0, x))^2 dx .
}
On the other hand, $v(0, x)$ is orthogonal to $f_1, f_2, \ldots, f_{N - 1}$, and so $v$ can be written as an orthogonal sum $\sum_{j = N}^\infty \beta_j u_j$. We now consider the weak and strong versions separately.

For the weak version, we have $\mu_N > \mu_n$. Thus,
\formula{
 \form_{H,V}(v, v) & = \sum_{j = 1}^\infty \mu_j |\beta_j|^2 \ge \mu_N \sum_{j = N}^\infty |\beta_j|^2 = \mu_N \int_{\R^d} (v(0, x))^2 dx ,
}
a contradiction with~\eqref{eq:nodal:form:norm} (for necessarily $\form_{H,V}(v, v) > 0$). Therefore, $u_n$ has less than $N$ nodal parts, as desired.

For the strong version, let $M = \min \{ j \in \N : \mu_j > \mu_n \}$. By~\eqref{eq:nodal:form:norm},
\formula{
 0 & = \form_{H,V}(v, v) - \mu_n \int_{\R^d} (v(0, x))^2 dx = \sum_{j = N}^\infty (\mu_j - \mu_n) |\beta_j|^2 = \sum_{j = M}^\infty (\mu_j - \mu_n) |\beta_j|^2 .
}
Thus, $\beta_j = 0$ for all $j \ge M$, and so $g(x) = v(0, x)$ is a linear combination of eigenfunctions $f_N, f_{N + 1}, \ldots, f_{M - 1}$, all corresponding to the same eigenvalue~$\mu_n$. It follows that $g$ is itself an eigenfunction corresponding to the eigenvalue $\mu_n$, and so $\form_V(g, g) = \mu_n \|g\|_2^2 = \form_{H,V}(v, v)$. By Theorem~\ref{thm:min}, we have $v = \ext(g)$, and therefore, by Theorem~\ref{thm:ucp}, the set $\{(s, x) \in \hh : v(s, x) = 0\}$ has zero Lebesgue measure. However, $v$ is a linear combination of the nodal parts $v_1, v_2, \ldots, v_N$ of $u_n$. If $u_n$ had another nodal part, the set $\{(s, x) \in \hh : v(s, x) = 0\}$ would contain the corresponding nodal domain and thus would have non-empty interior. Therefore, $u_n$ has no more than $N$ nodal parts, as desired.
\end{proof}

As remarked in the introduction, a variant of Theorem~\ref{thm:ch} can be given for radial functions when the underlying potential is a radial functions. The proof of the following result is very similar to the proof of Theorem~\ref{thm:ch}, and thus it is omitted.

\begin{theorem}\label{thm:ch:rad}
Suppose that $V(x)$ is a radial confining potential and that $\exp(-t \psi(|\xi|^2))$ is integrable over $\R^d$ for any $t > 0$. Let $\mu_{\rad,n}$ denote the non-decreasing sequence of eigenvalues of $\psi(-\Delta) + V(x)$ that correspond to radial eigenfunctions. Let $u_{\rad,n}(t, |x|)$ be the harmonic extension $\ext(f_{\rad,n})$. Then the number of nodal parts of $u_n$ on $(0, \infty) \times (0, \infty)$ does not exceed:
\begin{enumerate}[label={\rm (\alph*)}]
\item $\max \{j \in \N : \mu_j = \mu_n\}$ in the general case (weak version);
\item $\min \{j \in \N : \mu_j = \mu_n\}$ if $A(s)$ is positive and locally Lipschitz (strong version).
\end{enumerate}
In particular, the number of nodal parts of the profile function of $f_{\rad,n}$ on $(0, \infty)$ does not exceed $2 n - 1$.
\end{theorem}

We remark that an analogous statement is true for eigenvalues that correspond to eigenfunctions of the form $f(|x|) v(x)$ for any given solid harmonic polynomial $v(x)$ (for example, $v(x) = x_j$, $v(x) = x_i^2 - x_j^2$ or $v(x) = x_i x_j$, where $i \ne j$). For a detailed discussion, we refer to Section~2.1 in~\cite{bib:dkk17} or Appendix~C.3 in~\cite{bib:fls15}.

%
%

\section{Estimates of eigenvalues}
\label{sec:eig}

We continue to assume that $V(x)$ is a confining potential, and that $\psi$ is an unbounded complete Bernstein function. Our goal in this section is to compare the eigenvalues $\mu_n$ of $\op + V(x) = \psi(-\Delta) + V(x)$ with the eigenvalues $\lambda_n$ of $-\Delta + \gamma V(x)$ for a suitable $\gamma > 0$. This is achieved by constructing appropriate test functions and inserting them into the variational formula.

Throughout this section we write $\tscalar{V f, f}$ for $\int_{\R^d} V(x) |f(x)|^2 dx$.

Let $\gamma > 0$ and $\lambda > 0$. Let $f$ be weakly differentiable with $\nabla f \in \leb^2(\R^d)$ and $\|f\|_2 = 1$, and let $\eta = \|\nabla f\|_2^2 + \gamma \tscalar{V f, f}$. Let $u(s, x) = \ph_\lambda(s) f(x)$, where $\ph_\lambda$ is the function described by Theorem~\ref{thm:krein}. Clearly, $u \in \dom(\form_H)$ and
\formula{
 \form_H(u, u) & = \int_0^\infty \bigl( (\ph_\lambda'(s))^2 \tnorm{f}_2^2 + A(s) ((\ph_\lambda(s))^2 \tnorm{\nabla f}_2^2\bigr) dt .
}
Recall that $\tnorm{f}_2^2 = 1$ and $\tnorm{\nabla f}_2^2 = \eta - \gamma \tscalar{V f, f}$. Thus,
\formula{
 \form_H(u, u) & = \int_0^\infty \bigl((\ph_\lambda'(s))^2 + (\eta - \gamma \tscalar{V f, f}) A(s) (\ph_\lambda(s))^2\bigr) ds .
}
By equality in~\eqref{eq:ph:min},
\formula{
 \form_H(u, u) & = \psi(\lambda) + (\eta - \lambda - \gamma \tscalar{V f, f}) \int_0^\infty A(s) (\ph_\lambda(s))^2 ds .
}
Suppose that $\lambda \ge \eta$ and that
\formula[eq:condition]{
 \gamma \int_0^\infty A(s) (\ph_\lambda(s))^2 dt & \ge 1 .
}
Then it follows that
\formula{
 \form_{H,V}(u, u) & = \form_H(u, u) + \tscalar{V f, f} \le \psi(\lambda) .
}
This essentially proves the following comparison result.

\begin{theorem}\label{thm:est}
Suppose that $V$ is a nonnegative confining potential and $\lambda > 0$. Choose
\formula{
 \gamma & = \expr{\int_0^\infty A(s) (\ph_\lambda(s))^2 ds}^{-1} ,
}
and let $\lambda_n$ be the non-decreasing sequence of eigenvalues of $-\Delta + \gamma V$. Let $n$ be the greatest index such that $\lambda_n \le \lambda$. Then the $n$-th smallest eigenvalue $\mu_n$ of $\psi(-\Delta) + V(x)$ is not greater than $\psi(\lambda)$.
\end{theorem}

\begin{proof}
If $f$ is in the linear span of the first $n$ eigenfunctions of $-\Delta + \gamma V$, then $\eta = \|\nabla f\|_2^2 + \gamma \tscalar{V f, f}$ satisfies $\eta \le \lambda_n \le \lambda$. Therefore, by the discussion preceding the statement of the theorem, the function $u(s, x) = \ph_\lambda(s) f(x)$ satisfies
\formula{
 \form_{H,V}(u, u) & \le \psi(\lambda) .
}
From the min-max variational characterisation of $\mu_n$, we conclude that $\mu_n \le \psi(\lambda)$.
\end{proof}


Apparently, the above theorem is new even for the fractional Laplace operator. In this case
\formula{
 a(t) & = \frac{2^{\alpha - 1} \Gamma(\tfrac{\alpha}{2})}{\Gamma(1 - \tfrac{\alpha}{2})} \, t^{1 - \alpha} , \\
 \ph_\lambda(t) & = \frac{2^{1 - \alpha/2}}{\Gamma(\tfrac{\alpha}{2})} \, (\sqrt{\lambda} t)^{\alpha/2} K_{\alpha/2}(\sqrt{\lambda} t) ,
}
and
\formula{
 \int_0^\infty a(t) (\ph_\lambda(t))^2 dt & = \frac{2^{\alpha/2} \lambda^{\alpha/2 - 1/2}}{\Gamma(1 - \tfrac{\alpha}{2})} \int_0^\infty (\sqrt{\lambda} t)^{1 - \alpha/2} K_{\alpha/2}(\sqrt{\lambda} t) dt \\
 & = \frac{2^{\alpha/2} \lambda^{\alpha/2 - 1}}{\Gamma(1 - \tfrac{\alpha}{2})} \int_0^\infty r^{1 - \alpha/2} K_{\alpha/2}(r) dr = \lambda^{\alpha/2 - 1} ,
}
see formula~6.561.16 in~\cite{bib:gr07}.

\begin{corollary}\label{cor:est:fr}
Suppose that $V$ is a nonnegative confining potential and $\lambda > 0$. Let $\lambda_n$ be the non-decreasing sequence of eigenvalues of $-\Delta + \lambda^{1 - \alpha/2} V$. Let $n$ be the greatest index such that $\lambda_n \le \lambda$. Then the $n$-th smallest eigenvalue $\mu_n$ of $(-\Delta)^{\alpha/2} + V(x)$ is not greater than $\lambda^{\alpha/2}$.
\end{corollary}

This can be specialised further when $V$ is homogeneous with degree $p > 0$. Indeed, in this case if $f$ is an eigenfunction of $-\Delta + V(x)$ with eigenvalue $\lambda$, then $f_a(x) = f(a x)$ satisfies $\Delta f_a(x) = a^2 \Delta f(a x) = a^2 V(a x) f(a x) - a^2 \lambda f(a x) = a^{2 + p} V(x) f_a(x) - a^2 \lambda f_a(x)$. Therefore, $f_a$ is an eigenfunction of $-\Delta + a^{2 + p} V(x)$ with eigenvalue $a^2 \lambda$. In other words, if $\gamma = a^{2 + p}$, then $f_a$ is the eigenfunction of $-\Delta + \gamma V(x)$ with eigenvalue $\gamma^{2 / (2 + p)} \lambda$.

\begin{corollary}\label{cor:est:fr:hom}
Suppose that $V$ is a locally bounded, positive (except at zero) potential which is homogeneous with degree $p > 0$. Let $\lambda_n$ be the non-decreasing sequence of eigenvalues of $-\Delta + V(x)$, and $\mu_n$ be the non-decreasing sequence of eigenvalues of $(-\Delta)^{\alpha/2} + V(x)$. Then
\formula{
 \mu_n & \le \lambda_n^{(2 + p) \alpha / (2 \alpha + 2 p)} .
}
\end{corollary}

\begin{proof}
Choose $\lambda > 0$. Then the $n$-th smallest eigenvalue of $-\Delta + \lambda^{1 - \alpha/2} V$ is $\lambda^{(2 - \alpha) / (2 + p)} \lambda_n$. This is equal to $\lambda$ if $\lambda^{1 - (2 - \alpha) / (2 + p)} = \lambda_n$, that is, $\lambda = \lambda_n^{(2 + p) / (\alpha + p)}$.
\end{proof}

%
%

\section{Probabilistic motivations}
\label{sec:fk}

We end this article with a brief discussion of our results from the point of view of probability theory.

\subsection{Lévy processes}

If $\psi$ is a complete Bernstein function (more generally, if $\psi$ is a Bernstein function), then $\psi(-\Delta)$ is the generator of a Lévy process $X(t)$ in $\R^d$, which can be obtained from the standard Brownian motion $B(t)$ by a random change of time. More precisely, the process $X(t)$ can be constructed as $B(S(t))$, where $S(t)$ is a subordinator (an increasing Lévy process) such that the Laplace transform of the distribution of $S(t)$ is $\exp(-t \psi(\lambda))$, and such that $B(t)$ and $S(t)$ are independent processes. By a simple calculation, the Fourier transform of the distribution of $S(t)$ is equal to $(2 \pi)^{-d/2} \exp(-t \psi(|\xi|^2))$. For more information about Lévy processes, generators and subordination, we refer to~\cite{bib:s99,bib:ssv12}.

The semi-group generated by $-\psi(-\Delta)$ is a Feller semi-group (that is, a strongly continuous semi-group of operators on $\cont_0(\R^d)$), and the operators $\exp(-t \psi(-\Delta))$ are strongly Feller (that is, they map bounded measurable functions into continuous ones) if and only if $\exp(-t \psi(|\xi|^2))$ is integrable for every $t > 0$. These properties are inherited by the semi-group of $-\psi(-\Delta) + V(x)$, provided that $V(x)$ is bounded below and locally bounded above. In fact a more general statement is true, it is sufficient to assume that $V(x)$ is in an appropriate Kato class, see, for example, Theorem~2.5 in~\cite{bib:dv00}.

\subsection{Traces of diffusions}

The operators given in~\eqref{eq:oph}, \eqref{eq:oph:t} and~\eqref{eq:oph:s} are generators of certain diffusions $Y(t)$ in $\hh$, reflected on the boundary. The Lévy process $X(t)$ discussed in the previous section can be obtained as the trace of $Y(t)$ on the boundary. More precisely, the set of times $\{t > 0 : Y(t) \in \partial \hh\}$ is equal to the range of some subordinator $R(t)$ (namely, the inverse local time of the first coordinate of $Y(t)$ at zero), and $X(t)$ is equal to the process $Y(R(t))$.

Note that the above interpretation of $X(t)$ coincides with the one given in the previous section (as a subordinate Brownian motion) only when $a(t)$ or $A(s)$ is constant, and $X(t)$ is the Cauchy process. Indeed, the last $d$ coordinates of $Y(t)$ are not independent from the process $R(t)$, unless $a(t) = 1$ or $A(s) = 1$, so $Y(R(t))$ is not equivalent to subordination of the last $d$ coordinates of $Y(t)$ using the subordinator $R(t)$.

We remark that traces of Markov processes on appropriate sets have been studied in general in terms of quadratic (Dirichlet) forms, see~\cite{bib:cfy06}, as well as in probabilistic context, see~\cite{bib:ksv11}.

%
%

\medskip

\subsection*{Acknowledgements}
The authors thank Moritz Ka{\ss}mann for stimulating discussions on the subject of this article and Rupert Frank for valuable comments to a preliminary version of the paper.

%
%

%
%

\appendix

%
%

\section{Krein's string theory}
\label{app:krein}

\subsection{Extension to general Krein's strings}

For simplicity, most results have been stated for regular Krein's strings. However, their extension to the general case is typically straightforward. In this section we discuss the necessary changes.

Recall that Krein's string is a locally finite nonnegative measure on $[0, R)$, where $R \in (0, \infty]$. As pointed out in Theorem~\ref{thm:krein}, if $R$ is finite and $(R - s) A(ds)$ is a finite measure, then we need to impose a Dirichlet boundary condition at $s = R$ in the definition of $\ph_\lambda$.

Theorem~\ref{thm:krein} in Section~\ref{sec:krein:thm} is already stated for general Krein's strings. Change of variable described in Section~\ref{sec:krein:change} is not applicable unless $A(ds)$ is a function. In this case $A(s)$ can still be defined on a finite interval only, as well as the corresponding coefficient $a(t)$. We return to this subject in part~\ref{app:finite}.

We now turn our attention to Section~\ref{sec:h}. Definition~\ref{def:formh:s} takes the following form.

\begin{definition}
\label{def:formh:s:g}
For a function $u$ on $\hh$, we let
\formula[eq:formh:s:g]{
 \form_H(u, u) & = \int_0^R \int_{\R^d} (\partial_s u(s, x))^2 dx ds + \int_{[0, R)} \int_{\R^d} |\nabla_x u(s, x)|^2 dx A(ds) ,
}
with $u \in \dom(\form_H)$ if and only if the following conditions are satisfied: $u(s, x)$ is weakly differentiable in $s$; $u(s, \cdot)$, as a function in $\leb^2(\R^d)$, depends continuously on $s$ (which asserts that~\eqref{eq:u:int} holds for all $S \in [0, R)$);
\formula{
 \int_{[s_0, s_1)} \int_{\R^d} |u(s, x)|^2 dx A(ds) < \infty
}
whenever $0 < s_0 < s_1 < R$; for $A(ds)$-almost all $s$ the weak gradient $\nabla_x u(s, x)$ exists, and the integrals in~\eqref{eq:formh:s:g} are finite; if $R$ is finite and $A$ is a finite measure, then in addition $u(s, \cdot)$ converges in $\leb^2(\R^d)$ to zero as $s \to R^-$.
\end{definition}

We emphasize that since the measure $A$ may fail to be absolutely continuous, we need to assume that formula~\eqref{eq:u:int} holds for all $S \in [0, R)$, or at least for almost all $S \in [0, R)$ with respect to both the Lebesgue measure and $A$.

It is somewhat easier to describe $\dom(\form_H)$ in terms of the Fourier variable: formula~\eqref{eq:formh:s:f} becomes
\formula[eq:formh:s:f:g]{
 \form_H(u, u) & = \int_0^R \int_{\R^d} (\partial_s \fourier_x u(s, \xi))^2 d\xi ds + \int_{[0, R)} \int_{\R^d} |\xi|^2 |\fourier_x u(s, \xi)|^2 d\xi A(ds) ,
}
and $u \in \dom(\form_H)$ if and only if: $\fourier_x u(s, \xi)$ is weakly differentiable in $s$; $u(s, \cdot)$, as a function in $\leb^2(\R^d)$, depends continuously on $s$; the integral in~\eqref{eq:formh:s:f:g} is finite; and if $R$ is finite and $A$ is a finite measure, then in addition $\fourier_x u(s, \cdot)$ converges in $\leb^2(\R^d)$ to zero as $s \to R^-$.

The notion of the harmonic extension $u = \ext(f)$ (Definition~\ref{def:ext}) remains the same in the general case; so does Theorem~\ref{thm:oph:s}. The proofs of Theorems~\ref{thm:forms:s} and~\ref{thm:min} require only notational changes: replacing $\int_0^\infty$ and $A(s) ds$ by $\int_{[0, R)}$ and $A(ds)$. Theorem~\ref{thm:harm} extends without modifications to the case of finite $R$ when the measure $(R - s) A(ds)$ is infinite. When $(R - s) A(ds)$ is a finite measure, one needs to additionally impose a Dirichlet-type condition that $\|u(s, \cdot)\|_2$ converges to $0$ as $s \to R^-$.

For reader's convenience, we re-state the above theorems in the general context discussed above.

\begin{theorem}
\label{thm:harm:G}
Suppose that $f \in \leb^2(\R^d)$ and $u = \ext(f)$. Then $u$ satisfies
\formula[eq:harm:G]{
 \partial_s^2 u(s, x) + A(ds) \Delta_x u(s, x) & = 0
}
in $\hh$ in the sense of distributions. Furthermore, $f \in \dom(\op)$ if and only if the limit in the definition of $\partial_s u(0, \cdot)$ exists in $\leb^2(\R^d)$, and in this case
\formula{
 \op f & = -\partial_s u(0, \cdot) = -\lim_{s \to 0^+} \partial_s u(s, \cdot) .
}
Finally, suppose that $u$ satisfies~\eqref{eq:harm:G} in the sense of distributions, $\|u(s, \cdot)\|_2$ is a bounded function of $s$, and in addition, when $R$ is finite and $(R - s) A(ds)$ is a finite measure, $u(s, \cdot)$ converges to $0$ in $\leb^2(\R^d)$ as $s \to R^-$. Then $u = \ext(f)$ for some $f \in \leb^2(\R^d)$.
\end{theorem}

\begin{theorem}
\label{thm:forms:G}
Let $f \in \leb^2(\R^d)$ and $u = \ext(f)$. Then $f \in \dom(\form)$ if and only if $u \in \dom(\form_H)$. If $v \in \dom(\form_H)$, $f(x) = v(0, x)$ and $u = \ext(f)$, then $f \in \dom(\form)$ and
\formula{
 \form_H(v, v) & \ge \form_H(u, u) = \form(f, f) .
}
Moreover, the space $\dom(\form_H)$ is a direct sum of $\dom_0$ and $\dom_\harm$, where
\formula{
 \dom_0 & = \{ u \in \dom(\form_H) : u(0, x) = 0 \text{ for almost all } x \in \R^d \} , \\
 \dom_\harm & = \{ \ext(f) : f \in \dom(\form) \}
}
are orthogonal to each other with respect to $\form_H$.
\end{theorem}

With these extensions, the applications given in Sections~\ref{sec:ch} and~\ref{sec:eig} extend immediately to general Krein's strings.

\subsection{Change of variable for strings of finite length}
\label{app:finite}

In Definition~\ref{def:formh:t} we can now include coefficients $a(t)$ such that $(a(t))^{-1}$ is integrable over $[0, \infty)$, and possibly defined on an interval $[0, r)$ only, where $r \in (0, \infty]$. If $\sigma(t)$ is bounded, then the coefficient $A(s)$ is only defined on $[0, R)$ for finite $R = \lim_{t \to r^-} \sigma(t)$.

If $a(t)$ is integrable over $[0, r)$, then $A(s)$ is integrable on $[0, R)$. In this case we can extend $A(s)$ so that $A(s) = 0$ for $s \ge R$, and we can still apply Lemma~\ref{lem:formh:sub}, provided that we assume that $u(R, \cdot)$ is the limit of $u(\sigma(t), \cdot) = \tilde{u}(t, \cdot)$ as $t \to r^-$, and $u(s, x) = u(R, x)$ for $s \ge R$. It is then easy to see that the variational description of the eigenvalues in terms of the form $\tilde{\form}_H$ in Section~\ref{sec:var} remains unaltered.

However, if we instead consider $A(s)$ to be a Krein's string with length $R$, then Definition~\ref{def:formh:t} requires an additional condition that $\tilde{u}(t, \cdot)$ converges in $\leb^2(\R^d)$ to zero as $t \to r^-$ in order that $\tilde{u} \in \dom(\tilde{\form}_H)$.

If $a(t)$ is not integrable on $[0, r)$, then no further conditions are necessary: if $(a(t))^{-1}$ is integrable and $\tilde{u} \in \dom(\tilde{\form}_H)$, then $\tilde{u}(t, \cdot)$ automatically converges to $0$ in $\leb^2(\R^d)$ as $t \to r^-$.

Note that the following seven combinations are possible:
\begin{enumerate}[label={\rm (\alph*)}]
\item\label{it:a} $r = \infty$, $(a(t))^{-1}$ not integrable ($R = \infty$), $a(t)$ not integrable ($A$ infinite);
\item\label{it:b} $r = \infty$, $(a(t))^{-1}$ not integrable ($R = \infty$), $a(t)$ integrable ($A$ finite);
\item\label{it:c} $r = \infty$, $(a(t))^{-1}$ integrable ($R < \infty$), $a(t)$ not integrable ($A$ infinite);
\item\label{it:d} $r < \infty$, $(a(t))^{-1}$ not integrable ($R = \infty$), $a(t)$ not integrable ($A$ infinite);
\item\label{it:e} $r < \infty$, $(a(t))^{-1}$ not integrable ($R = \infty$), $a(t)$ integrable ($A$ finite);
\item\label{it:f} $r < \infty$, $(a(t))^{-1}$ integrable ($R < \infty$), $a(t)$ not integrable ($A$ infinite);
\item\label{it:g} $r < \infty$, $(a(t))^{-1}$ integrable ($R < \infty$), $a(t)$ integrable ($A$ finite),
\end{enumerate}
and only case~\ref{it:d} seems to be rather exotic: examples corresponding to other cases are given in Section~\ref{sec:ex}. More specifically, the classical and Caffarelli--Silvestre extension problems for $(-\Delta)^{\alpha/2}$ (see Sections~\ref{sec:ex:classical} and~\ref{sec:ex:cs}) correspond to~\ref{it:a}. The extension problem for $(m^2 - \Delta)^{\alpha/2} - m$ (see Section~\ref{sec:ex:relat}) is an example of type~\ref{it:b}, while the one for \mbox{$(m^2 - \Delta)^{\alpha/2} + m$} (see Section~\ref{sec:ex:relat:dual}) belongs to~\ref{it:c}. The operators related to the theory of linear water waves (see Section~\ref{sec:ex:waves}), with Fourier symbols $|\xi| \tanh(R |\xi|)$ and $|\xi| (\tanh(R |\xi|))^{-1}$ are clearly of type~\ref{it:g}. Finally, the example described in Section~\ref{sec:ex:nonstandard} is of type~\ref{it:e} or~\ref{it:f}, depending on the value of the parameter $\alpha$.

\subsection{Existence and properties of $\ph_\lambda$}

For completeness, we recall some basic facts from Krein's spectral theory of strings. Throughout this section $A(ds)$ is a Krein's string, that is, a locally finite measure on $[0, R)$, with possibly infinite $R > 0$.

\begin{theorem}[see~\cite{bib:kw82}]
\label{thm:krein:kw}
Let $\lambda \ge 0$. The Cauchy problem
\formula{
 \begin{cases}
 f''(s) = \lambda f(s) A(ds) , \\
 f(0) = a, \\
 f'(0) = b
 \end{cases}
}
has a unique solution on $[0, R)$ for any $a, b \in \R$. Here the second derivative is understood in the sense of distributions, that is, we require that
\formula{
 f(S) & = f(0) + f'(0) S + \int_{[0, S)} (S - s) \lambda f(s) A(ds)
}
for $s \in [0, R)$.

If $f_N$ and $f_D$ are the solutions of the above problem with $f_N(0) = f_D'(0) = 1$, $f_N'(0) = f_D(0) = 0$, then
\formula[eq:krein:kw]{
 \lim_{s \to R^-} \frac{f_D(s)}{f_N(s)} & = \int_0^R \frac{1}{(f_N(s))^2} \, ds = \frac{1}{\psi(\lambda)}
}
for a complete Bernstein function $\psi$.

Conversely, for any complete Bernstein function $\psi$ there is a unique corresponding Krein's string $A(ds)$, so that \emph{Krein's correspondence} is a bijection between Krein's strings and complete Bernstein functions.
\end{theorem}

Note that the first equality in~\eqref{eq:krein:kw} follows from the fact that the Wrońskian $f_D' f_N - f_N' f_D$ is constant one, and therefore $1 / f_N^2 = (f_D / f_N)'$. In particular, if $\lambda \ge 0$, then $f_D / f_N$ is non-decreasing. Many further properties of Krein's correspondence, including the relationship between asymptotic behaviour of $A(ds)$ and $\psi(\lambda)$, can be found in~\cite{bib:kw82}; see also the references therein.

The function $\ph_\lambda$ satisfying the conditions of Theorem~\ref{thm:krein} can be constructed as
\formula{
 \ph_\lambda(s) & = f_N(s) - (\psi(\lambda))^{-1} f_D(s) .
}
Clearly, $\ph_\lambda''(s) = \lambda \ph_\lambda(s) A(ds)$, $\ph_\lambda(0) = 0$ and $-\ph_\lambda'(0) = \psi(\lambda)$. Furthermore, since $f_D / f_N$ is non-decreasing, $f_D(s) / f_N(s) \le \psi(\lambda)$, and therefore $\ph_\lambda(s) \ge 0$ for all $s$. It follows that $\ph_\lambda''(s) \ge 0$, so that $\ph_\lambda$ is convex. It is also easy to see that $\ph_\lambda$ is decreasing. Indeed, it is the limit as $c \to (\psi(\lambda))^{-1}$ of $f_N - c f_D$. However, if $c > (\psi(\lambda))^{-1}$, then $f_N - c f_D$ is nonnegative and convex on some initial interval $[0, s_0)$ and then it becomes negative and concave on $(s_0, R)$, and thus it is non-increasing. It follows that $\psi_\lambda$ is necessarily non-increasing.

Observe that the estimate~\eqref{eq:ph:min} (with equality if and only if $f = \ph_\lambda$) extends to the general setting: one only needs to replace the usual integration by parts by a suitable application of Fubini's theorem.

For $\lambda = 0$ we simply have $f_N(s) = 1$ and $f_D(s) = s$, so that $\ph_0(s) = 1 - \psi(\lambda) s$. Suppose that $\lambda > 0$. If $R$ is infinite, then $f_D(s) \ge s$, so that $f_D$ is unbounded. It follows that $\ph_\lambda$ is the only bounded solution of the equation $f''(s) = \lambda f(s) A(ds)$ with $f(0) = 1$. Similarly, if $R$ is finite, then $f_N$ is unbounded if and only if $(R - s) A(ds)$ is an infinite measure. If this is the case, then again $\ph_\lambda$ is the only bounded solution of $f''(s) = \lambda f(s) A(ds)$ such that $f(0) = 1$. Otherwise, all solutions of this problem are bounded.

Similarly, one can show that if $\lambda > 0$ and $A$ is an infinite measure, then $\ph_\lambda$ is the only solution of the problem $f''(s) = \lambda f(s) A(ds)$ for which $f'$ is bounded. If $A$ is a finite measure, however, then all solutions of this problem have bounded derivative.

Finally, we note that by equality in~\eqref{eq:ph:min}, for $\lambda > 0$, $(\ph_\lambda(s))^2$ is integrable over $[0, R)$ with respect to $A(ds)$. Furthermore, if $R$ is infinite or $A$ is an infinite measure, $\ph_\lambda$ is the only solution of $f''(s) = \lambda f(s) A(ds)$ such that $f(0) = 1$ and $\int_{[0, R)} (f(s))^2 A(ds) < \infty$. Otherwise the integrability condition is satisfied by all solutions of $f''(s) = \lambda f(s) A(ds)$. This explains the need for additional Dirichlet boundary condition in Definition~\ref{def:formh:s:g}.

\end{document}